\documentclass[a4paper,12pt,reqno]{amsart}
\usepackage[left=2cm,right=2cm,top=2.5cm,bottom=3cm]{geometry}
\usepackage[foot]{amsaddr}
\usepackage{amsmath,tensor}
\usepackage{dsfont}
\usepackage{enumitem}
\usepackage{color}
\usepackage{booktabs,tabulary}
\usepackage{graphicx,xcolor}
\usepackage{todonotes}

\usepackage[
	bookmarks,
	bookmarksopen,
	bookmarksopenlevel=0,
	bookmarksnumbered,
	breaklinks,
	ocgcolorlinks,
	colorlinks=true,
	citecolor=blue,
	linktoc=all
	]{hyperref}

\usepackage{tikz}
\usetikzlibrary{arrows.meta}
\definecolor{plum}{rgb}{0.36078, 0.20784, 0.4}
\definecolor{chameleon}{rgb}{0.30588, 0.60392, 0.023529}
\definecolor{cornflower}{rgb}{0.12549, 0.29020, 0.52941}
\definecolor{scarlet}{rgb}{0.8, 0, 0}
\definecolor{brick}{rgb}{0.64314, 0, 0}
\definecolor{sunrise}{rgb}{0.80784, 0.36078, 0}
\definecolor{lightblue}{rgb}{0.15,0.35,0.75}
\definecolor{carolina}{RGB}{153, 186, 221}
\tikzstyle{axisarrow} = [-{Latex[inset=0pt,length=5pt]}]

\newtheorem{theorem}{Theorem}
\newtheorem{proposition}{Proposition}

\newtheorem{lemma}{Lemma}
\newtheorem{fact}{Fact}
\newtheorem{definition}{Definition}
\newtheorem{example}{Example}

\newcommand{\lie}{\ensuremath{\mathcal L}}
\newcommand{\RR}{\ensuremath{\mathbb{R}}}
\newcommand{\CC}{\ensuremath{\mathbb{C}}}
\newcommand{\by}{\ensuremath{\times}}
\newcommand{\del}{\ensuremath{\partial}}
\newcommand{\diff}[2]{\ensuremath{\frac{\del #1}{\del #2}}}
\newcommand{\projalg}{\ensuremath{\mathfrak{p}}}
\newcommand{\homalg}{\ensuremath{\mathfrak{h}}}
\newcommand{\veps}{\ensuremath{\varepsilon}}
\newcommand{\barz}{\ensuremath{\bar z}}
\newcommand{\tr}{\ensuremath{\mathop{tr}}}
\newcommand{\Id}{\ensuremath{\mathds{1}}}

\usepackage[backend=bibtex,
			bibencoding=ascii,
			giveninits=true,
			url=false,
			isbn=false,
			doi=false,
			style=numeric,
			maxnames=5]{biblatex}
\begin{document}
	
\title{The c-projective symmetry algebras of K\"ahler surfaces}
\author{G.~Manno*, J.~Schumm*, A.~Vollmer*$^\dagger$}

\email{giovanni.manno@polito.it, jsmaths@gmx.net, andreas.vollmer@uni-hamburg.de}

\address{* DISMA, Politecnico di Torino, Corso Duca degli Abruzzi 24, 10129 Torino, Italy}
\address{$^\dagger$ FB Mathematik, Universit\"at Hamburg, Bundesstra{\ss}e 55, 20146 Hamburg, Deutschland}

\begin{abstract}
Let $M$ be a K\"ahler manifold  with complex structure $J$ and K\"ahler 
metric $g$. A c-projective vector field is a vector field on $M$ whose flow 
sends $J$-planar curves to $J$-planar curves, where $J$-planar curves are 
analogs of what (unparametrised) geodesics are for pseudo-Riemannian 
manifolds (without complex structure). The c-projective symmetry algebras of 
K\"ahler surfaces with essential (i.e., non-affine) c-projective vector 
fields are computed.
\end{abstract}

\maketitle

\section{Introduction}

Studying infinitesimal transformations of $J$-planar curves is a natural
question that has received increased attention during the last two decades; a good summary is \cite{CEMN2020}.
Historically, the question can be traced back to classical mathematicians
like, for instance, Lie \cite{lie_1882,lie_1883}, Painlevé 
\cite{painleve_1894}, Beltrami \cite{beltrami_1865} and Levi-Civita 
\cite{levi-civita_1896} who investigated transformations of the differential 
equations that describe trajectories of Hamiltonian systems up to 
reparametrisations.

An infinitesimal symmetry of a system of ordinary differential equations
(ODEs) is a vector field on the space of dependent and independent variables
whose local flow sends solutions to solutions. It is well known that the
set of infinitesimal symmetries of a given ODE has the structure of a Lie
algebra \cite{olver2000applications}. Studying the symmetry Lie algebras for ODEs is a powerful technique
by which the equivalence of two ODEs can be investigated, i.e.,~whether one
can be mapped into the other by a change of variables. Of course, the dimension of the
symmetry Lie algebras of equivalent ODEs is the same.
For instance, take the ODEs (we denote derivatives by subscripts)
\begin{equation}\label{eqn:ode.2.dim}
	y_{xx}=-e^{-2x}y_x^3
\end{equation}
and
\begin{equation}\label{eqn:ode.3.dim}
	y_{xx}=\frac12 y_x-\frac12 e^{-2x}y_x^3\,.
\end{equation}
They are not equivalent as~\eqref{eqn:ode.2.dim} has a 2-dimensional
symmetry Lie algebra generated by
\begin{equation*} 
	\partial_y\quad\text{and}\quad\partial_x+y\partial_y\,,
\end{equation*}
while \eqref{eqn:ode.3.dim} has a 3-dimensional symmetry Lie algebra
generated by
\begin{equation*} 
	\partial_y\,,\quad \partial_x+y\partial_y
	\quad\text{and}\quad y\partial_x+\frac12 y^2\partial_y\,.
\end{equation*}
In other words, there is no change of variables $(x,y)\to
(x_\text{new}=x_\text{new}(x,y),y_\text{new}=y_\text{new}(x,y))$ sending
\eqref{eqn:ode.2.dim} into
\eqref{eqn:ode.3.dim}. Furthermore, neither \eqref{eqn:ode.2.dim} nor
\eqref{eqn:ode.3.dim} can be mapped, by a change of variables, to
\begin{equation}\label{eqn:ode.8.dim}
	y_{xx}=0
\end{equation}
as this latter ODE admits an 8-dimensional Lie algebra of symmetries, whose
generators are
\begin{equation}\label{eqn:8.dim.Lie.alg}
	\partial_x\,,\quad\partial_y\,,\quad x\partial_x\,,\quad
	x\partial_y\,,\quad y\partial_x\,,\quad y\partial_y\,,\quad 
	x^2\partial_x+xy\partial_y\quad \text{and}\quad 
	xy\partial_x+y^2\partial_y\,.
\end{equation}
These examples are, of course, not chosen randomly, but are geometrically
motivated. Indeed, \eqref{eqn:ode.2.dim}, \eqref{eqn:ode.3.dim}
and~\eqref{eqn:ode.8.dim} are the equations of ``unparametrised geodesics''
for certain Riemannian metrics.
Specifically, \eqref{eqn:ode.8.dim} is the equation of \emph{lines} of
$\mathbb{R}^2$, i.e.~the geodesic curves of the Euclidean metric $dx^2+dy^2$.
Generators of the \emph{projective algebra} of $\mathbb{R}^2$ are given by
\eqref{eqn:8.dim.Lie.alg}: the local flow of any generator of
\eqref{eqn:8.dim.Lie.alg} sends lines of $\mathbb{R}^2$ into lines of
$\mathbb{R}^2$.
Likewise, Equation \eqref{eqn:ode.2.dim} is realised as the equation of
unparametrised geodesics for the metric $e^{4x}dx^2+e^{2x}dy^2$, Equation
\eqref{eqn:ode.3.dim} is realised as the equation of unparametrised
geodesic of $e^{3x}dx^2+e^xdy^2$. It follows, as a byproduct, that the
metrics $e^{4x}dx^2+e^{2x}dy^2$, $e^{3x}dx^2+e^{x}dy^2$ and $dx^2+dy^2$ are
pairwise non-isometric.

Let us return to~\eqref{eqn:ode.8.dim}.
By a suitable change of variables, \eqref{eqn:ode.8.dim} transforms into an ODE describing segments of \emph{great circles} of the sphere. This insight is due to
Liouville~\cite{liouville_1889} and illustrated by Figure~\ref{fig:beltrami}.
In fact, up to changing variables, the ODE~\eqref{eqn:ode.8.dim} is the
equation of unparametrised geodesics of any 2-dimensional metric with
constant curvature.
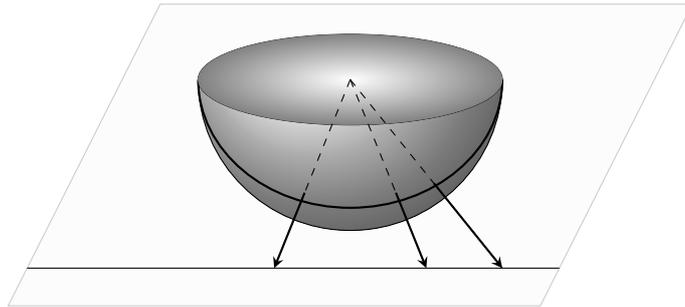
\begin{figure}\label{fig:beltrami}
	\centering
	\begin{tikzpicture}
		\begin{scope}[xslant=0.5]
			\filldraw[fill=gray!10,opacity=0.2] (-3,1) -- (4,1) -- (4,-3) -- (-3,-3) -- cycle;
			\draw[thin] (-3,-2.5) -- (4,-2.5);
		\end{scope}
		\begin{scope}
			\clip(-3,-3) -- (-3,0) -- (-2,0) arc (180:360:2cm and 0.6cm) -- (3,0) -- (3,-3) -- (-3,-3);
			\filldraw[ball color=white] (0,0) circle [radius=2cm];
		\end{scope}
		\draw[black,thick] (2,0) arc [start angle=0,end angle=-180,x radius=2cm,y radius=17mm];
		\draw[solid] (2,0) arc [start angle=0,end angle=180,x radius=2cm,y radius=6mm];
		\draw (2,0) arc [start angle=0,end angle=-180,x radius=2cm,y radius=6mm];
		\shade[shading=radial, inner color=white, outer color=white!50!black, opacity=0.70] (2,0) arc (0:360:2cm and
		0.6cm);
		\draw[black,dashed] (0,0) -- (2,-2.5);
		\draw[black,dashed] (0,0) -- (1,-2.5);
		\draw[black,dashed] (0,0) -- (-1,-2.5);
		\draw[black,thick,-stealth] (1.1,-1.375) -- (2,-2.5);
		\draw[black,thick,-stealth] (0.6,-1.5) -- (1,-2.5);
		\draw[black,thick,-stealth] (-0.6,-1.5) -- (-1,-2.5);
	\end{tikzpicture}
	\caption{The gnomonic (central) projection maps great circles to lines.}
\end{figure}
\medskip

In the present paper we ask a very similar question, yet in a different
context. Before we were speaking about geodesics and (pseudo-)Riemannian metrics (i.e., of arbitrary signature).
Now we consider $J$-planar curves and K\"ahler metrics (of arbitrary signature) instead. The analog of projective symmetry algebras, in this context, are the so-called c-projective algebras (see below for a proper introduction). Our aim is to obtain, explicitly, the full, non-trivial\footnote{By ``non-trivial'' we mean the existence of an essential (non-affine)
	c-projective vector field, see below Def.~\ref{def:c-projective}.} c-projective 
algebras of K\"ahler surfaces.

\subsection{Preliminaries}

Let $(M,J,g)$ be a K\"ahler manifold of real dimension $2n$ and arbitrary
signature. This means that $M$ is a differentiable manifold with
(pseudo-)Riemannian metric $g$ and complex structure $J$, such that $\nabla
J=0$, where $\nabla$ is the Levi-Civita connection of $g$, and the 2-form
$\omega$ on $M$ defined by $\omega(X,Y)=g(J(X),Y)$ is a symplectic form. In
this context, the concept of unparametrised geodesics is replaced by that of
$J$-planar curves.
\begin{definition}
Let $(M,J,g)$ be a K\"ahler manifold. Let $\nabla$ be its Levi-Civita
connection. A curve $\gamma:I\subseteq\mathbb{R}\to \gamma(t)\in M$ such that
$$
	\nabla_{\dot\gamma}\dot\gamma\wedge\dot\gamma\wedge J\dot\gamma = 0
$$
is called a \emph{$J$-planar curve}.
\end{definition}
\noindent
It is easily verified that this definition is independent of parametrisation.
The following definition captures the corresponding infinitesimal symmetries.
\begin{definition}\label{def:c-projective}
Let $(M,J,g)$ be a K\"ahler manifold. A vector field is called \emph{c-projective} if its local flow maps
$J$-planar curves into $J$-planar curves.
\end{definition}
It is well known that the c-projective vector fields of a K\"ahler
manifold form a Lie algebra, see for instance~\cite{CEMN2020}, which we denote by $\projalg(g)$, suppressing $J$ in the notation.
\emph{Affine} vector fields, i.e.\ vector fields that preserve the Levi-Civita connection of~$g$ and the complex structure~$J$, are simple examples of c-projective vector fields. 
C-projective vector fields that are not affine are called \emph{essential}.
We need to introduce two more concepts before we are able to state the main fact that serves as point of departure for the present paper.
\begin{definition}
	Let $g,\hat g$ be two K\"ahler metrics on the same complex manifold $(M,J)$. They are called \emph{c-projectively equivalent} if they share the same $J$-planar curves.
	%
	The equivalence class of all K\"ahler surfaces that are c-projectively
	equivalent is called a \emph{c-projective K\"ahler structure} and denoted by $(M,J,[g])$.
\end{definition}
\noindent Provided $n\geq2$, a c-projective vector field automatically preserves the complex structure, c.f.~\cite{BMMR2015}.
We also recall the definition of holomorphic sectional curvature (HSC), which
is a restriction of the usual sectional curvature to $J$-invariant planes.
\begin{definition}
	Let $g$ be a K\"ahler metric with Riemann curvature $R$.
	The HSC of a $J$-invariant plane $\Pi=\langle v,Jv\rangle$ in $p\in M$ is 
	defined by
	$$ K(\Pi) = \frac{R(v,Jv,v,Jv)}{||v||^4}\,, $$
	where $v\in T_pM$ is a vector at $p$.
	If the map $\Pi\mapsto K(\Pi)$ is constant and independent of $p$, then 
	we say that $g$ has	constant HSC.	
\end{definition}

\noindent In \cite{BMMR2015} a non-sharp local description is obtained for c-projective K\"ahler structures admitting essential c-projective vector fields.
We remark that the list does not claim to be sharp, and that indeed some parameter configurations would not lead to a metric. Such configurations are tacitly excluded from now on without further mentioning.

\begin{fact}[\cite{BMMR2015}]\label{fct:metrics}
	Let $(M,J,g_0)$ be a K\"ahler surface of non-constant HSC and with an
	essential c-projective vector field.
	Then, in a neighborhood of almost every point of $M$, there are local
	coordinates such that a certain metric $g\in[g_0]$ and its K\"ahler form $\omega$ defined by
	$\omega(X,Y)=g(J(X),Y)$ are as in one of the cases below. 
	\begin{enumerate}[label=(\roman*)]
		\item	\textbf{Liouville type:}
				In the new coordinates $(x_0,x_1,s_0,s_1)$,
				\begin{align*}
					g &= (\rho_0-\rho_1)(F_0^2dx_0^2+\veps F_1^2dx_1^2) \\
						&\quad +\frac1{\rho_0-\rho_1}\left[
								\left(\frac{\rho_0'}{F_0}\right)^2(ds_0+\rho_1ds_1)^2
								+\veps \left(\frac{\rho_1'}{F_1}\right)^2(ds_0+\rho_0ds_1)^2
							\right]
					\\
					\omega &= \rho_0'\,dx_0\wedge(ds_0+\rho_1\,ds_1)
								+\rho_1'\,dx_1\wedge(ds_0+\rho_0\,ds_1)
				\end{align*}
				where $\veps\in\{\pm1\}$ and where the univariate
				functions $\rho_0,\rho_1,F_0,F_1$ are given as in one of the
				cases below.
				\begin{enumerate}[label=(L\arabic*)]
					\item $\rho_i(x_i)=x_i$, $F_i=c_i\in\RR$
					\item $\rho_i(x_i)=c_ie^{(\beta-1)x_i}$, $F_i=d_ie^{-\frac12(\beta+2)x_i}$ with $c_i,d_i\in\RR$
							($\beta\ne1$)
					\item $\rho_i(x_i)=x_i$, $F_i=c_ie^{-\frac32x_i}$, $c_i\in\RR$
					\item $\rho_i(x_i)=-\tan(x_i)$, $F_i=\frac{c_ie^{-\frac32\beta x_i}}{\sqrt{|\cos(x_i)|}}$,
							  $c_i\in\RR$
				\end{enumerate}
		\item	\textbf{Complex type:} 
				In the new coordinates $(z,\bar z,s_0,s_1)$,
				\begin{align}
					\label{eqn:C.metrics}
					g &= \frac14\,(\bar\rho-\rho)(F^2dz^2-\bar F^2d\barz^2) \\
					\nonumber
						&\quad +\frac4{\rho-\bar\rho}\left[
								\left(\frac{\bar\rho'}{\bar F}\right)^2(ds_0+\rho ds_1)^2
								-\left(\frac{\rho'}{F}\right)^2(ds_0+\bar\rho ds_1)^2
						\right]
					\\
					\nonumber
					\omega &= \rho'\,dz\wedge(ds_0+\bar\rho\,ds_1)
						+\bar\rho'\,d\barz\wedge(ds_0+\rho\,ds_1)
				\end{align} 
				where the holomorphic functions $\rho,F$ are given as in one
				of the cases below (for brevity we denote
				$\rho'=\diff{\rho}{z}$ and $\bar\rho'=\diff{\bar\rho}{\barz}$).
				\begin{enumerate}[label=(C\arabic*)]
					\item $\rho(z)=z$,
					$F=\varsigma=\varsigma_0+i\varsigma_1\in\CC$
					\item $\rho(z)=e^{(\beta-1)z}$,
					$F=\varsigma e^{-\frac12(\beta+2)z}$ with
					$\varsigma=\varsigma_0+i\varsigma_1\in\CC$ ($\beta\ne1$)
					\item $\rho(z)=z$, $F=\varsigma e^{-\frac32z}$,
					$\varsigma=\varsigma_0+i\varsigma_1\in\CC$
					\item $\rho(z)=-\tan(z)$, $F=\frac{\varsigma
					e^{-\frac32\beta z}}{\sqrt{\cos(z)}}$,
							$\varsigma=\varsigma_0+i\varsigma_1\in\CC$
				\end{enumerate}
		\item	\textbf{Degenerate type:}
				In the new coordinates $(x_0,x_1,s_0,s_1)$,
				\begin{align*}
					g &= -\rho\,h +\rho F^2 dx_0^2 +\frac1{\rho}\left(\frac{\rho'}{F}\right)^2\theta^2
					\\
					\omega &= -\rho\,d\tau +\rho' dx_0\wedge\theta
				\end{align*}
				where $h=\sum_{i,j}h_{ij}(s_0,s_1)ds_ids_j$ and 
				$\theta=dx_1-\tau$, and
				where the functions $\rho,F,\tau$ are given as in one of the
				cases below.
				\begin{enumerate}[label=(D\arabic*)]
					\item $\rho(x_0)=\frac1{x_0}$, 
					$F(x_0)=\frac{c_1}{\sqrt{|x_0|}}$
					and $\tau=s_0ds_1$. Moreover, $h=G(s_1)ds_0^2+\frac1{G(s_1)}ds_1^2$.
					\item[(D2a)] $\rho(x_0)=c_1e^{-3x_0}$, $F(x_0)=d_1$,
								$c_1,d_1\in\RR$, $\tau=s_0ds_1$, $h=G(s_1)ds_0^2+\frac1{G(s_1)}ds_1^2$.
					\item[(D2b)] $\rho(x_0)=c_1e^{(\beta-1)x_0}$, 
					$F(x_0)=d_1e^{-\frac12(\beta+2)x_0}$
								($\beta\ne1$, $\beta\ne-2$),
								$c_1,d_1\in\RR$,\\
								$\tau=-\frac1{\beta+2}e^{-(\beta+2)s_0}G(s_1)ds_1$,
								$h=e^{-(\beta+2)s_0}G(s_1)(ds_0^2+ds_1^2)$
					\item[(D3)] $\rho(x_0)=\frac1{x_0}$, 
					$F(x_0)=\frac{c_1e^{-\frac32x_0}}{\sqrt{|x_0|}}$,
								$\tau=-\frac13e^{-3s_0}G(s_1)ds_1$, $h=e^{-3s_0}G(s_1)(ds_0^2+ds_1^2)$
				\end{enumerate}
	\end{enumerate}
\end{fact}
\noindent
To avoid misconceptions, we would like to point out a few particulars of the
list in Fact~\ref{fct:metrics}:
\begin{itemize}[left=1em]
	\item
	Not all the metrics listed in Fact~\ref{fct:metrics} are of non-constant holomorphic sectional curvature. Theorem~\ref{thm:CHSC} below is
	addressing this point.
	\item
	Not every metric of non-constant holomorphic curvature with an essential
	c-projective vector field is isometric to one explicitly stated in the list, but every such metric can be expressed by using a metric $g$ from the list and a
	suitable companion metric $\hat g$ (detailed in
	Appendix~\ref{app:companion.metrics}) using the formula
	\begin{equation}\label{eqn:metric.construction}
		g[t_1,t_2] = \frac{\left(
						t_1|\det(g)|^{\tfrac16}g^{-1}
						+t_2|\det(\hat g)|^{\tfrac16}\hat g^{-1}\right)^{-1}
					}{
						\sqrt{\det\left(
							t_1|\det(g)|^{\tfrac16}g^{-1}
							+t_2|\det(\hat g)|^{\tfrac16}\hat g^{-1}
						\right)}
					} \,,\quad t_i\in\mathbb{R}\,.
	\end{equation}
	The problem of describing these individual metrics, up to isometries, is
	not solved in the current paper, but is going to be addressed in a
	separate paper.
	For our purposes here it is sufficient to state the following: The three types in Fact~\ref{fct:metrics}
	are disjoint, i.e., a metric satisfying the hypothesis of
	Fact~\ref{fct:metrics} is either of Liouville, of complex or of degenerate
	type.
	\item
	Metrics obtained using Formula \eqref{eqn:metric.construction} have the same c-projective algebra as
	the metric~$g$ itself. In particular, for metrics of constant holomorphic
	sectional curvature, the algebra of c-projective vector fields is
	well understood and always isomorphic to that of the Fubini-Study metric (see Example~\ref{ex:fubini-study} below).
	For the metrics of non-constant HSC which admit a non-affine c-projective 
	vector field, we determine the full c-projective algebras in this paper.
	\item
	Fact~\ref{fct:metrics} provides a characterisation modulo c-projective 
	equivalence and indeed there are metrics in the list of 
	Fact~\ref{fct:metrics} that admit only c-projective vector fields that 
	are affine (e.g., in the case L1). However, generically, the metrics that are c-projectively equivalent 
	to one listed in Fact~\ref{fct:metrics} admit essential c-projective 
	vector fields. 
	Note 
	that c-projective algebras are invariant under c-projective 
	transformations.
\end{itemize}
\noindent Before turning to our main results, let us consider the case of
constant HSC more closely. Such metrics are c-projectively
equivalent to the Fubini-Study metric \cite{FKMR2010,kruglikov_2016}, and therefore their c-projective algebras are isomorphic. 
\begin{example}[Fubini-Study metric on $\CC\mathbb{P}^2$]
\label{ex:fubini-study}
Let $z_1=x+iy$ and $z_2=s+it$ be complex coordinates on $\CC\mathbb{P}^2$.
The Fubini-Study metric on $\CC\mathbb{P}^2$ is locally described by
\begin{equation}\label{eqn:Fubini.metric.local} g=\frac{(1+\sum_iz_i\bar{z}_i)\sum_jdz_jd\bar{z}_j-\sum_{i,j}\bar{z}_jz_idz_jd\bar{z}_i}{(1+\sum_iz_i\bar{z}_i)^2}\,.
\end{equation}
The complex structure $J$ in these coordinates takes the standard form
\begin{equation*} 
	J=
	\left(
	\begin{array}{cccc}
		0 & 1 & 0 & 0
		\\
		-1 & 0 & 0 & 0
		\\
		0 & 0 & 0 & 1
		\\
		0 & 0 & -1 & 0
	\end{array}
	\right)
\end{equation*}
The c-projective algebra of~\eqref{eqn:Fubini.metric.local} is isomorphic
to~$\mathfrak{sl}(3,\CC)$ \cite{MS1983,FKMR2010}.
An explicit realisation of this algebra in terms of 16 generating vector
fields can be found in Appendix~\ref{app:fubini-study}.
\end{example}

\noindent
With this example in mind, the purpose of the present paper is twofold: We first identify the metrics in Fact~\ref{fct:metrics} that have
constant HSC, and then we determine the full
c-projective algebras of those metrics in Fact~\ref{fct:metrics} that have
non-constant HSC.
In particular, we obtain the following four main theorems.
\begin{theorem}\label{thm:CHSC}
	Among the metrics in Fact~\ref{fct:metrics} the following have constant
	HSC:
	\begin{enumerate}
		\item[(L1)] Metrics of L1 type have constant HSC if $\varepsilon=-1$
		and $c_1=\pm c_0$.
		\item[(L2)] Metrics of L2 type have constant HSC if
		\begin{itemize}
			\item[(i)]
			$\beta=-\tfrac12$
			and $c_1=\varepsilon c_0\frac{d_1^2}{d_0^2}$,
			\item[(ii)]
			$\beta=-2$,
			$\varepsilon=-1$ and $d_1=\pm d_0$.
		\end{itemize}
		\item[(C1)] Metrics of C1 type have constant HSC if either
		$\Re(\varsigma)=0$
		or $\Im(\varsigma)=0$.
		\item[(C2)] Metrics of C2 type have constant HSC if
		\begin{itemize}
			\item[(i)] $\beta=-\tfrac12$
			and either $\Re(\varsigma)=0$ or $\Im(\varsigma)=0$,
			\item[(ii)]
			$\beta=-2$
			and either $\Re(\varsigma)=0$ or $\Im(\varsigma)=0$.
		\end{itemize} 
		\item[(D1)] Metrics of D1 type have constant HSC if the metric $h$
		has
		vanishing Gau{\ss} curvature.
		\item[(D2a)] Metrics of D2a type have constant HSC if the metric $h$
		has
		Gau{\ss} curvature exactly $-\frac{9}{d_1^2}$.
	\end{enumerate}
	All other metrics in Fact~\ref{fct:metrics} are of non-constant HSC.
\end{theorem}
\noindent
Having answered the question which metrics of Fact~\ref{fct:metrics} have 
constant HSC, we now turn to determining the c-projective vector fields of 
those metrics with non-constant HSC.
We begin with Liouville type metrics.
\begin{theorem}\label{thm:Liouville.main}
	Let $g$ be a K\"ahler surface covered by Fact~\ref{fct:metrics}. Assume
	it is of non-constant HSC and that it is of Liouville type.
	Then its c-projective algebra is isomorphic to one of the following.
	\begin{enumerate}[label=(L\arabic*)]
		\item 
		The c-projective algebra is 4-dimensional,
		\begin{align*}
			\projalg(g) = \big\langle\quad
				\del_{x_0}+\del_{x_1}-s_1\del_{s_0}\,,\quad
				x_0\del_{x_0}+x_1\del_{x_1}+2s_0\del_{s_0}+s_1\del_{s_1}\,,\quad
				\del_{s_0}\,,\quad\del_{s_1}\quad\big\rangle\,.
		\end{align*}
		\item 
		The c-projective algebra is either 3-dimensional or 4-dimensional.
		\begin{enumerate}[label=(\roman*)]
			\item If $\beta\ne0$,
			then the c-projective algebra is 3-dimensional,
			\begin{align*}
				\projalg(g) = \big\langle\quad
					\del_{x_0}+\del_{x_1}-(\beta+2)s_0\del_{s_0}
							-(2\beta+1)s_1\del_{s_1}\,,\quad
					\del_{s_0}\,,\quad\del_{s_1}\quad\big\rangle\,.
			\end{align*}
			\item If $\beta=0$,
			then the c-projective algebra is 4-dimensional,
			\begin{align*}
				\projalg(g) = \big\langle\quad
						\del_{x_0}+\del_{x_1}
							-2s_0\del_{s_0}
							-s_1\del_{s_1}\,,\quad
						\frac1{c_0}e^{x_0}\del_{x_0}
							+\frac1{c_1}e^{x_1}\del_{x_1}
							+s_1\del_{s_0}\,,\quad
						\del_{s_0}\,,\quad\del_{s_1}\quad\big\rangle\,.
			\end{align*}
		\end{enumerate}
		\item 
		The c-projective algebra is 3-dimensional,
		\begin{align*}
			\projalg(g) = \big\langle\quad
					\partial_{x_0}+\partial_{x_1}
							-(3s_0+s_1)\partial_{s_0}
							-3s_1\partial_{s_1}\,,\quad
					\del_{s_0}\,,\quad\del_{s_1}\quad\big\rangle\,.
		\end{align*}
		\item 
		The c-projective algebra is 3-dimensional,
		\begin{align*}
			\projalg(g) = \big\langle\quad
					\partial_{x_0}+\partial_{x_1}
							-(3\beta s_0-s_1)\partial_{s_0}
							-(s_0+3\beta s_1)\partial_{s_1}\,,\quad
					\del_{s_0}\,,\quad\del_{s_1}\quad\big\rangle\,.
		\end{align*}
	\end{enumerate}
\end{theorem}
\noindent
Next, we consider complex type metrics.
\begin{theorem}\label{thm:complex.main}
	Let $g$ be a K\"ahler surface covered by Fact~\ref{fct:metrics}. Assume
	it is of non-constant HSC and that it is of complex type.
	Then its c-projective algebra is isomorphic to one of the following.
	\begin{enumerate}[label=(C\arabic*)]
		\item 
		The c-projective algebra is 4-dimensional:
		\begin{align*}
			\projalg(g) = \big\langle\quad
					\del_z+\del_{\barz}-s_1\del_{s_0}\,,\quad
					z\del_z+\barz\del_{\barz}
							+2s_0\del_{s_0}+s_1\del_{s_1}\,,\quad
					\del_{s_0}\,,\quad\del_{s_1}\quad\big\rangle\,.
		\end{align*}
		\item 
		The c-projective algebra is either 3-dimensional or 4-dimensional,
		\begin{enumerate}[label=(\roman*)]
			\item If $\beta\ne0$,
			then the c-projective algebra is 3-dimensional,
			\begin{align*}
				\projalg(g) = \big\langle\quad
						\del_{z}+\del_{\barz}
							-(\beta+2)s_0\del_{s_0}
							-(2\beta+1)s_1\del_{s_1}\,,\quad
						\del_{s_0}\,,\quad\del_{s_1}\quad\big\rangle\,.
			\end{align*}
			\item If $\beta=0$,
			then the c-projective algebra is 4-dimensional,
			\begin{align*}
				\projalg(g) = \big\langle\quad
					\del_z+\del_{\barz}-2s_0\del_{s_0}-s_1\del_{s_1}\,,\quad
					e^{z}\del_{z}+e^{\barz}\del_{\barz}+s_1\del_{s_0}\,,\quad
					\del_{s_0}\,,\quad\del_{s_1}\quad\big\rangle\,.
			\end{align*}
		\end{enumerate}
		\item 
		The c-projective algebra is 3-dimensional,
		\begin{align*}
			\projalg(g) = \big\langle\quad
					\partial_{z}+\partial_{\barz}-(3s_0+s_1)\partial_{s_0}
							-3s_1\partial_{s_1}\,,\quad
					\del_{s_0}\,,\quad\del_{s_1}\quad\big\rangle\,.
		\end{align*}
		\item 
		The c-projective algebra is 3-dimensional,
		\begin{align*}
			\projalg(g) = \big\langle\quad
			\partial_z+\partial_{\barz}
			-(3\beta s_0-s_1)\partial_{s_0}
			-(s_0+3\beta s_1)\partial_{s_1}\,,\quad
			\del_{s_0}\,,\quad\del_{s_1}\quad\big\rangle\,.
		\end{align*}
	\end{enumerate}
\end{theorem}
\noindent
Lastly, there are the metrics of degenerate type.
\begin{theorem}\label{thm:degenerate.main}
	Let $g$ be a K\"ahler surface covered by Fact~\ref{fct:metrics}. Assume
	it is of non-constant HSC and that it is of degenerate type.
	Then its c-projective algebra is isomorphic to one of the following.
	\begin{enumerate}[label=(D\arabic*)]
		\item 
		The c-projective algebra is of dimension 3, 4 or 5.
		\begin{enumerate}[label=(\roman*)]
			\item If $h$ has constant (non-zero) Gau{\ss} curvature,
			$$ h = G(s_1)ds_0^2+\frac{ds_1^2}{G(s_1)}\,,\qquad
			G(s_1) = \kappa s_1^2+\mu_1s_1+\mu_2\,,\quad
			\kappa\ne0\,, $$
			then the c-projective algebra is 5-dimensional (see
			Section~\ref{sec:D1.results}).
			\item If $h$ admits a 2-dimensional homothetic symmetry algebra,
			 then
			 $$ h = G(s_1)ds_0^2+\frac{ds_1^2}{G(s_1)}\,,\qquad
			 	G(s_1) = \kappa(\mu_1s_1+\mu_2)^{\frac{2(\mu_1+1)}{\mu_1}} $$
			 with $\projalg(h)=\langle\del_{s_0}\,,~~
			 		(\mu_1+2)s_0\del_{s_0}-(\mu_1s_1+\mu_2)\del_{s_1}\rangle$.
			 Then the c-projective algebra is 4-dimensional,
			 \begin{align*}
			 	\projalg(g) &= \big\langle\quad
			 			\partial_{x_0}\,,\quad
			 			\partial_{x_1}\,,\quad
			 			s_1\partial_{x_1}+\partial_{s_0}\,,
			 			\\
			 			&\qquad\quad
				 		2x_0\partial_{x_0}+2x_1\partial_{x_1}
				 		+(\mu_1+2)s_0\partial_{s_0}
			 			-(\mu_1s_1+\mu_2)\partial_{s_1}\quad
				 	\big\rangle\,.
			 \end{align*}
			\item If neither of the previous two cases is assumed, then the
			c-projective algebra is 3-dimensional,
			\begin{align*}
				\projalg(g) = \big\langle\quad
					\del_{x_0}\,,\quad\del_{x_1}\,,\quad
					s_1\del_{x_1}+\del_{s_0}\quad\big\rangle\,.
			\end{align*}
		\end{enumerate}
		\item[(D2a)]
		The c-projective algebra is of dimension 3 or 5.
		\begin{enumerate}[label=(\roman*)]
			\item If the metric $h$ has constant Gau{\ss} curvature,
			then the c-projective algebra of $g$ is 5-dimensional (see Section~\ref{sec:D2a.results}).
			\item If $h$ admits a 2-dimensional or 1-dimensional homothetic
			symmetry algebra, then the c-projective algebra is 3-dimensional,
			\begin{align*}
				\projalg(g) = \big\langle\quad
				\del_{x_0}\,,\quad\del_{x_1}\,,\quad
				s_1\del_{x_1}+\del_{s_0}\quad\big\rangle\,.
			\end{align*}
		\end{enumerate}
		\item[(D2b)]
		The c-projective algebra is of dimension 2, 3 or 5.
		\begin{enumerate}[label=(\roman*)]
			\item If $h$ has constant Gau{\ss} curvature, then it is already
			flat. The c-projective algebra of $g$ thus is 5-dimensional (see
			Section~\ref{sec:D2b.results}).
			\item If $h$ admits a 2-dimensional homothetic
			symmetry algebra, then the c-projective algebra is 3-dimensional (see
			Section~\ref{sec:D2b.results}).
			\item If $h$ has 1-dimensional homothetic algebra, then the
			c-projective algebra of $g$ is 2-dimensional,
			$$ \projalg(g) = \big\langle\quad
					\del_{x_0}+(\beta+2)x_1\del_{x_1}+\del_{s_0}\,,\quad
					\del_{x_1}\quad\big\rangle\,. $$
		\end{enumerate}
		\item[(D3)]
		The c-projective algebra is of dimension 2, 3 or 5.
		\begin{enumerate}[label=(\roman*)]
			\item If $h$ has constant Gau{\ss} curvature, then $h$ is already
			flat. The c-projective algebra of $g$ is thus
			5-dimensional (see Section~\ref{sec:D3.results}).
			\item If $h$ admits a 2-dimensional homothetic
			symmetry algebra, then the c-projective algebra is 3-dimensional (see
			Section~\ref{sec:D3.results}).
			\item If $h$ has a 1-dimensional homothetic algebra, then the
			c-projective algebra of $g$ is 2-dimensional,
			$$ \projalg(g) = \big\langle\quad
				\del_{x_0}+3x_1\del_{x_1}+\del_{s_0}\,,\quad
				\del_{x_1}\quad\big\rangle\,. $$
		\end{enumerate}
	\end{enumerate}		
\end{theorem}

\noindent In the remainder of the paper we shall prove
Theorems~\ref{thm:CHSC}, \ref{thm:Liouville.main}, \ref{thm:complex.main}
and~\ref{thm:degenerate.main}.
The sections are organised as follows:
we begin by proving Theorem~\ref{thm:CHSC} in Section~\ref{sec:CHSC}.
Section~\ref{sec:method} outlines the methodology we apply for the proofs of
Theorems~\ref{thm:Liouville.main}, \ref{thm:complex.main}
and~\ref{thm:degenerate.main}. The proofs are then completed in
Sections~\ref{sec:Liouville},~\ref{sec:complex} and~\ref{sec:degenerate},
respectively.
The paper is complemented by an appendix, where we summarise some background
material that will likely help readers not yet familiar with the literature,
particularly~\cite{BMMR2015}. However, the appendix is not needed to
follow the main text of the paper.

\section{Constant holomorphic sectional curvature}\label{sec:CHSC}

\noindent
Theorem~\ref{thm:CHSC} identifies those Kähler metrics in 
Fact~\ref{fct:metrics} that are of constant HSC. In the current section we 
prove this theorem.
It should be noted that K\"ahler manifolds of constant HSC are, as such, well
understood, see for example~\cite{FKMR2010,BMMR2015,kruglikov_2016,CEMN2020}. The following fact is well-established.
\begin{fact}\label{fct:CHSC}
	Let $g$ be a K\"ahler surface metric ($n=2$) of constant HSC. Then the c-projective
	algebra of $g$ is (locally) isomorphic to that of the Fubini-Study metric 
	outlined in Example~\ref{ex:fubini-study}.
\end{fact}
\noindent
For higher dimensions, a similar statement holds true.
In the theory of c-projective K\"ahler structures, the degree of mobility is an important attribute: informally speaking, the degree of mobility is
the minimal number of parameters needed to describe the c-projective class of a given K\"ahler metric $g$. 
Let $[g]$ be a c-projective K\"ahler structure. Then it can be shown that
there is a minimal number $N$ such that $[g]$ is parametrised by $N$ linearly
independent K\"ahler metrics $g_k$ ($1\leq k\leq N$) via
\[
	[g] = \left\{\quad
		\frac{
			\left(\sum_k t_k|\det(g_k)|^{\tfrac16}g_k^{-1}\right)^{-1}
		}{
			\sqrt{\det\left(\sum_k t_k|\det(g_k)|^{\tfrac16}g_k^{-1}\right)}
		}\quad : \quad
		t_k\in\RR
		\quad\right\}.
\]
This number $D([g])=N$ is called the degree of mobility of $[g]$. We also say
that $g\in[g]$ has degree of mobility $D(g)=N$.
\begin{lemma}[\cite{MD1978,MS1983,BMMR2015}]\label{la:mobility}~
	
	\noindent(i) Let $g$ be a K\"ahler metric of complex dimension~$n$. It is of constant HSC if and 
	only if $D(g)=(n+1)^2$.
	
	\noindent(ii) Let $g$ be a K\"ahler surface metric ($n=2$) of non-constant HSC, which
	admits an essential c-projective vector field. Then $D(g)=2$.
\end{lemma}
\begin{proof}
	Claim~(i) is proven in~\cite{MD1978,MS1983} and claim~(ii)
	in~\cite{BMMR2015}. See also~\cite{MR2015}.
\end{proof}

For what follows here and in the following sections, we also need a
certain (1,1)-tensor field.
In order to introduce it, let $g$ and $\hat g$ be c-projectively equivalent
metrics (not necessarily of dimension $2n=4$, but in the following we will
only consider K\"ahler surfaces). We define
\begin{equation}\label{eqn:L}
	L[g,\hat g] := \left|\frac{\det(\hat g)}{\det(g)}\right|^{\frac1{2(n+1)}}\,
	\hat g^{-1}g\,,
\end{equation}
see also Appendix~\ref{app:companion.metrics}.
We proceed in two steps in order to prove Theorem~\ref{thm:CHSC}:
\begin{enumerate}
	\item
		In the light of Lemma~\ref{la:mobility}, in the present section we 
		are looking for metrics listed in Fact~\ref{fct:metrics} whose degree 
		of mobility is larger than two.
		A necessary criterion is obtained as follows: If the degree of 
		mobility of any of the metrics in Fact~\ref{fct:metrics} is larger
		than two, then~$L$ satisfies the (special) Sinjukov
		equations\footnote{We remark that the converse is not true: not every 
		$V_n(B)$ geometry has degree of mobility larger than two.} of a
		so-called $V_4(B)$ geometry \cite{Mikes1998},
		\begin{subequations}\label{eqn:sinjukov}
		\begin{align}
			\label{eqn:sinjukov.1}
			\nabla_k L_{ij} &= g_{ik}\Lambda_{j} + g_{jk}\Lambda_{i}
										+\omega_{ik}\Lambda_{a}J\indices{^a_j}
										+\omega_{jk}\Lambda_{a}J\indices{^a_i}
			\\
			\label{eqn:sinjukov.2}
			\nabla_{j}\Lambda_{i} &= \mu g_{ij}+BL_{ij}
			\\
			\label{eqn:sinjukov.3}
			\nabla_k\mu &= 2B\Lambda_{k}
		\end{align}
		\end{subequations}
		with $\Lambda_{k}=\nabla_k\Lambda$ and $\Lambda=\frac14\tr(L)$, and where $B$ is a constant.
		We first determine necessary conditions for the parameters
		$\beta,\varsigma,c_i,d_i$. Indeed, if $g$ has constant HSC, it needs 
		to satisfy
		this system of partial differential equations (PDEs), where~$L$ is 
		computed from $g$ and a suitable companion metric as explained in 
		Appendix~\ref{app:companion.metrics}.
	\item
		For the special parameter configurations identified in the first
		step, we can check the HSC of the metrics directly, deciding if $g$ 
		has constant HSC and thus completing the proof.
\end{enumerate}
Since in the first step we merely need a necessary criterion, and since the
different types of metrics in Fact~\ref{fct:metrics} are qualitatively
different, we split the discussion according to the type of the metrics and
discuss each case separately.

\subsection{Cases L1--L4}
We proceed according to the outlined steps.
The metrics $g$ considered in the present section are those of Fact~\ref{fct:metrics} that are within the cases L1 to L4.
To obtain the necessary condition from Sinjukov's special PDE system, we consider~\eqref{eqn:sinjukov.2} locally as a matrix equation.
Its $(2,3)$-entry\footnote{Counting components starts at $0$ here.}
yields a condition of the form
$$ \alpha+B\beta = 0\,, $$
where $B\in\mathds R$, and where $\alpha$ and $\beta$ are functions depending on $\rho_0,\rho_1,F_0,F_1,\varepsilon$.
Provided $\beta\ne0$ (which is fulfilled in all cases discussed below), we may solve for $B$. We differentiate once w.r.t.~$x_0$ and obtain a necessary condition for $g$ having degree of mobility larger than two, $D(g)\geq3$.

Subsequently, in the second step, we check by a direct computation if, with the described necessary condition 
fulfilled, the metric has indeed constant HSC, or not.
The results are detailed in the following table.\medskip

\begin{center}
\begin{tabulary}{\linewidth}{l|p{5cm}|p{7cm}}
	Metric & Sinjukov condition & constant HSC? \\
	\toprule\toprule
	L1 & $c_0^2\varepsilon+c_1^2=0$ & Yes, if $\varepsilon=-1$ and $c_1=\pm c_0$ \\
	\midrule
	L2 & $\beta=-\frac12$, $c_1d_0^2\varepsilon+c_0d_1^2=0$
		& Yes, if $\varepsilon=\pm1$ and $c_1=-\varepsilon c_0\frac{d_1^2}{d_0^2}$ \\
	\cmidrule{2-3}
		& $\beta=-2$, $d_0^2\epsilon+d_1^2=0$
		& Yes, if $\varepsilon=-1$ and $d_1=\pm d_0$ \\
	\midrule
	L3 & Never possible & No \\
	\midrule
	L4 & $c_0^2\varepsilon+c_1^2=0$ & No
\end{tabulary}
\end{center}
In the last column of the table, ``yes'' indicates that the Liouville metrics of the respective type are of constant HSC if the stated conditions are met. ``No'' indicates that there do not exist such metrics for the respective case.

\subsection{Cases C1--C4}
The reasoning is identical to that in the cases C1--C4, with the formal replacements ($\varsigma=\varsigma_0+i\varsigma_1$)
\begin{itemize}
	\item \textit{Cases C1, C3 and C4:} $c_0\to\varsigma$, $c_1\to\bar\varsigma$ and $\varepsilon\to-1$,
	\item \textit{Case C2:} $\varepsilon=-1$ and $c_0\to1$, $c_1\to1$ and $d_0\to\varsigma$ , $d_1\to\bar\varsigma$.
\end{itemize}  
We arrive at the following table (recall $\varsigma=\varsigma_0+i\varsigma_1$):\medskip

\begin{center}
\begin{tabulary}{\linewidth}{l|p{6cm}|p{8cm}}
	Metric & Sinjukov condition & constant HSC? \\
	\toprule\toprule
	C1 & $\varsigma^2=\bar\varsigma^2$ & Yes, if $\varsigma_0=0$ or
	$\varsigma_1=0$ \\
	\midrule
	C2 & $\beta=-\frac12$, $\varsigma^2=\bar\varsigma^2$
		& Yes, if $\beta=-\frac12$ and either $\varsigma_0=0$ or $\varsigma_1=0$\\
	\cmidrule{2-3}
		& $\beta=-2$, $\varsigma^2=\bar\varsigma^2$ 
		& Yes, if $\beta=-2$ and either $\varsigma_0=0$ or $\varsigma_1=0$ \\
	\midrule
	C3 & Never possible & No \\
	\midrule
	C4 & $\varsigma^2=\bar\varsigma^2$ & No
\end{tabulary}
\end{center}

\subsection{Cases D1--D3}
Consider~\eqref{eqn:sinjukov.2} locally as a matrix equation.
The $(1,3)$-entry of the matrix equation gives
$$
	BF^3\rho^3+2\rho F'\rho'+2F\rho'^2-2F\rho\rho'' = 0
$$
for some constant $B\in\RR$.
Solving for $B$ (note that $F\rho\neq 0$) and differentiating once w.r.t.~$x_0$, we obtain an
explicit condition depending on $\rho$, $F$ and their derivatives only:
\[
3\rho^2\,F'^2\rho' - F\rho^2\,F''\rho'
+ 4F\rho\,F'\rho'^2 + 3F^2\rho'^3
- 3F\rho^2\,F'\rho'' - 4F^2\rho\,\rho'\rho''
+ F^2\rho^2\,\rho''' = 0\,.
\]
Evaluating this condition for each of the cases D1--D3, and then examining 
for constant HSC, we arrive at the following table.\medskip

\begin{center}
\begin{tabulary}{\linewidth}{l|p{6cm}|p{6cm}}
	Metric & Sinjukov condition & constant HSC? \\
	\toprule\toprule
	D1 & The constraint for higher mobility is trivial, i.e.~no constraint
		& Yes, if $h$ has vanishing Gau{\ss} curvature \\
	\midrule
	D2a & The constraint for higher mobility is trivial, i.e.~no constraint
		& Yes, if $h$ has Gau{\ss} curvature $-\frac9{d_1^2}$  \\
	\midrule
	D2b & Would require $\beta=-2$ or $\beta=1$, which are forbidden
		& No \\
	\midrule
	D3 & Never possible & No
\end{tabulary}
\end{center}

We have therefore proven Theorem~\ref{thm:CHSC}.

\section{Method for Theorems~\ref{thm:Liouville.main}, \ref{thm:complex.main}
and \ref{thm:degenerate.main}}\label{sec:method}

We now turn to the main purpose of this paper and compute the c-projective
algebras of the metrics of non-constant HSC in Fact~\ref{fct:metrics}.
Methodologically, we use an approach similar to that employed in
\cite{BMMR2015}, \cite{solodovnikov_1956} and \cite{MV2022}.
Lemma~\ref{la:mobility} ensures that the degree of mobility is precisely two
if a metric from Fact~\ref{fct:metrics} has non-constant HSC.
Metrics with constant HSC have, of course, already been identified in
Theorem~\ref{thm:CHSC}.
The method will subsequently be adapted for the individual cases of Liouville,
complex and degenerate type in the following three sections.
Many of the mentioned computations were performed and verified using the sagemath computer algebra software~\cite{sagemath}.
The following lemma is useful, see e.g.~\cite{bolsinov_2011,BMR2021}.
\begin{lemma}
	Let $g$ and $\hat g$ be c-projectively equivalent K\"ahler surfaces which
	are
	not of constant HSC.
	Let $v$ be a c-projective vector field of $g$.
	Then there are constants $a_{ij}\in\RR$ ($1\leq i,j\leq 2$) such that
	\begin{align}
		\lie_vL &= -a_{01}L^2+(a_{11}-a_{00})L+a_{10}\Id
		\label{eqn:LvL}
		\\
		\lie_vg &= -a_{00}(n+1)\,g-a_{01}\,\left(gL+\tfrac12\tr(L)g\right)\,.
		\label{eqn:Lvg}
	\end{align}
	Particularly, if $v$ is essential, then $a_{01}\ne0$, and if $v$ is properly homothetic (i.e.~homothetic, but not Killing) for $g$, then $a_{01}=0$ and $a_{00}\ne0$. If $v$ is
	Killing for $g$, then $a_{00}=a_{01}=0$.
\end{lemma}

We prove the main theorems by solving the equations~\eqref{eqn:LvL}
and~\eqref{eqn:Lvg}.
A partial solution of~\eqref{eqn:LvL} is well known and can be exploited in
the following. It concerns eigenvalues of $L$ and is partly folklore. An 
explicit proof can be found in~\cite{BMR2021,MV2022}.
\begin{lemma}\label{la:vf}
	Let $g$ and $\hat g$ be c-projectively equivalent K\"ahler surfaces which
	are not of constant HSC, and let $v$ be a
	c-projective vector field.
	Any eigenvalue $f$ of~\eqref{eqn:L} satisfies
	$$
		v(f) = -a_{01}f^2+(a_{11}-a_{00})f+a_{10}\,.
	$$
\end{lemma}
Our study below is going to be local, and therefore the
equations~\eqref{eqn:LvL} and~\eqref{eqn:Lvg} can be represented by $4\by4$
matrices. In fact, for the metrics of Fact~\ref{fct:metrics} these matrices
have a block structure consisting of four $2\by2$ matrices each.
\begin{align*}
	\text{Eq.~\eqref{eqn:LvL}:}\qquad &&
	\lie_vL+a_{01}L^2-(a_{11}-a_{00})L-a_{10}\Id&=0
	\qquad\leftrightarrow&
	\begin{pmatrix} (*a) & (*b) \\ (*c) & (*d) \end{pmatrix}=0
	\\
	\text{Eq.~\eqref{eqn:Lvg}:}\qquad &&
	\lie_vg  +a_{00}(n+1)\,g+a_{01}\,\left(gL+\tfrac12\tr(L)g\right)&=0
	\qquad\leftrightarrow&
	\begin{pmatrix} (*A) & (*B) \\ (*C) & (*D) \end{pmatrix}=0
\end{align*}

\section{Proof of Theorem~\ref{thm:Liouville.main} (Liouville type)}
\label{sec:Liouville}

We proceed in two steps. First we prove the following auxiliary
result, which allows us to reduce the problem to the univariate realm.
\begin{lemma}\label{la:Liouville}
	Any c-projective symmetry of a Liouville metric in the coordinates of
	Fact~\ref{fct:metrics} is of the form
	\[
		v = v^0(x_0)\partial_{x_0}+v^1(x_1)\partial_{x_1}
			+v^2(s_0,s_1)\partial_{s_0}+v^3(s_0,s_1)\partial_{s_1}\,.
	\]
	In particular, if $F_1^2\rho_0'^2+\varepsilon F_0^2\rho_1'^2\ne0$,
	then even
	\begin{align*}
		v^2(s_0,s_1) &= \mu s_0-a_{10}s_1+\nu
		\\
		v^3(s_0,s_1) &= -a_{01}s_0+(\mu+a_{00}-a_{11})s_1+\nu_0
	\end{align*}
	for constants $\mu,\nu$ and $\nu_0$.
\end{lemma}
Using this lemma, the proof of Proposition~\ref{prop:degenerate} is then
going to continue along the following lines:
The remaining algebraic conditions are of the form
\begin{equation}\label{eqn:liouville.cases.algebraic.condition}
	\sum_\mu\sum_\nu\ c_{\mu\nu} F_\mu(x_0)\hat F_\nu(x_1) = 0
\end{equation}
where the $c_{\mu\nu}$ depend on the parameters $a_{ij}$ and on the 
parameters of the metric.
Without losing generality, we may take the functions $F_\mu$ (respectively
$\hat F_\nu$) to be linearly independent functions (otherwise use linear
dependence to replace some functions until linear independence is 
achieved, changing the functions $F_\mu$ and $\hat F_\nu$ 
in~\eqref{eqn:liouville.cases.algebraic.condition}).
We hence conclude
\[
	c_{\mu\nu} = 0 \quad\forall\ \mu,\nu\,,
\]
and then solve all remaining conditions, arriving eventually at
Theorem~\ref{thm:Liouville.main}.

In the following three subsections, we first prove Lemma~\ref{la:Liouville}, then discuss the exceptional case
$$ F_1^2\rho_0'^2+\varepsilon F_0^2\rho_1'^2=0\,, $$
excluded in Lemma~\eqref{la:Liouville}, in Section~\ref{sec:liouville-exceptional}. The proof of Theorem~\ref{thm:Liouville.main} is then completed in Section~\ref{sec:liouville-finalise}.

\subsection{Proof of Lemma~\ref{la:Liouville}}
We begin the proof by computing the tensor $L$.
For the metric $g$ we have, in the coordinates $(x_0,x_1,s_0,s_1)$,
$$
	L = \begin{pmatrix}
			\rho_0(x_0) &&& \\
			& \rho_1(x_1) && \\
			&& \rho_0+\rho_1 & \rho_0\rho_1 \\
			&& -1 & 0
		\end{pmatrix}
$$
and take the c-projective vector field as
$
	v = (v^0,v^1,v^2,v^3)
$
where each component is a function $v^i=v^i(x_0,x_1,s_0,s_1)$ in four
variables.
We first take the (0,0) and (1,1) components of~$(*a)$, i.e.\ which are the
equations from Lemma~\ref{la:vf}.
Since $\rho_0$ and $\rho_1$ are non-constant, we find ($j\in\{0,1\}$)
\begin{align*}
	v^j &= -\frac{
				a_{01}\rho_j(x_j)^2 + (a_{00}-a_{11})\rho_j(x_j) - a_{10}
			}{
				\frac{\partial\rho_j}{\partial x_j}
			}
\end{align*}
and conclude $v^0=v^0(x_0)$ as well as $v^1=v^1(x_1)$.
Resubstituting into~\eqref{eqn:LvL}, the block $(*b)$ is redundant, and the
block $(*c)$ yields
\begin{align*}
	\rho_0(x_0)\frac{\partial v^3}{\partial x_0}
	+\frac{\partial v^2}{\partial x_0} = 0
	\\
	\rho_1(x_1)\frac{\partial v^3}{\partial x_1}
	+\frac{\partial v^2}{\partial x_1} = 0
\end{align*}
Moreover, then the block $(*B)$ implies
\begin{align*}
	\frac{\partial\rho_0(x_0)}{\partial x_0}
	\frac{\partial v^3}{\partial x_0} = 0
	\\
	\frac{\partial\rho_1(x_1)}{\partial x_1}
	\frac{\partial v^3}{\partial x_1} = 0
	\\
	(\rho_0-\rho_1)\,
	\frac{\partial\rho_1}{\partial x_1}
	\frac{\partial v^3}{\partial x_1} = 0
\end{align*}
We conclude $v^3=v^3(s_0,s_1)$ and then $v^2=v^2(s_0,s_1)$.

Now revisit the block $(*d)$. From the (3,3) and (3,2) components, we obtain
\begin{align*}
	\left(a_{01}+\frac{\partial v^3}{\partial s_0}\right)\rho_0\rho_1
	+\left(a_{10}+\frac{\partial v^2}{\partial s_1}\right) &= 0
	\\
	\left(b+\frac{\partial v^3}{\partial s_0}\right)(\rho_0+\rho_1)
	+\left(a_{00}-a_{11}+\frac{\partial v^2}{\partial s_0}
			-\frac{\partial v^3}{\partial s_1}\right) &= 0
\end{align*}
and conclude
\begin{align*}
	\frac{\partial v^3}{\partial s_0} &= -a_{01}
	\\
	\frac{\partial v^2}{\partial s_1} &= -a_{10}
	\\
	\frac{\partial v^3}{\partial s_1}-\frac{\partial v^2}{\partial s_0}
	&= a_{00}-a_{11}\,.
\end{align*}
With this additional information, we move on to the block $(*D)$. After
differentiating the (2,2) component w.r.t.~$s_0$ and $s_1$, respectively,
we have
\begin{align*}
	(F_1^2\rho_0'^2+\varepsilon F_0^2\rho_1'^2)\,
	\frac{\del^2v^2}{\del s_0^2} &= 0 \\
	(F_1^2\rho_0'^2+\varepsilon F_0^2\rho_1'^2)\,
	\frac{\del^2v^2}{\del s_0\del s_1} &= 0
\end{align*}
Assuming
\begin{equation}\label{eqn:exceptional.criterion.L}
	F_1^2\rho_0'^2+\varepsilon F_0^2\rho_1'^2\ne0\,,
\end{equation}
we infer
\[
	\frac{\del^2v^2}{\del s_0^2}=0
	\qquad\text{and}\qquad
	\frac{\del^2v^2}{\del s_0\del s_1}=0\,.
\]
Therefore
\[
	v^2(s_0,s_1) = \mu s_0-a_{10}s_1+\nu\qquad
	(\mu,\nu\in\RR)
\]
and then
\[
	v^3(s_0,s_1) = -a_{01}s_0+(\mu+a_{00}-a_{11})s_1+\nu_0\qquad
	(\mu,\nu_0\in\RR)\,.
\]
For this result we have assumed~\eqref{eqn:exceptional.criterion.L}. The next 
section is dedicated to the complementary case.

\subsection{The exceptional branch of Lemma~\ref{la:Liouville}}\label{sec:liouville-exceptional}

According to Lemma~\ref{la:Liouville}, an exceptional situation occurs if the 
condition~\eqref{eqn:exceptional.criterion.L} is not satisfied. In this case 
we have
\begin{equation}\label{eqn:L.exceptional}
	F_1^2\rho_0'^2+\varepsilon F_0^2\rho_1'^2 = 0\,.
\end{equation}
As $F_i$, $\rho_i$ and $\varepsilon$ are real and non-vanishing, and 
$\varepsilon^2=1$, we conclude that~\eqref{eqn:L.exceptional} implies
$$ \varepsilon=-1. $$
We therefore have to consider the exceptional branch arising if
\begin{equation}\label{eqn:L.exceptional.2}
	F_1^2\rho_0'^2-F_0^2\rho_1'^2 = 0
\end{equation}
holds. We are now going to study this condition for the metrics L1--L4.

\subsubsection{Metrics L1}\label{sec:exceptional.L1}
We begin with the metrics of L1 type. We have
\[
	\rho_i=x_i\,,\quad
	F_i = c_i\in\RR.
\]
From the previous discussion we also infer ($t,\tau\in\RR$)
\[
	a_{00}=\frac35t\,,\quad
	a_{01}=0\,,\quad
	a_{11}=-\frac25t\,,\qquad
	a_{10}=\tau\,,\qquad
	\mu = -2t\,.
\]
Condition~\eqref{eqn:L.exceptional.2} thus becomes
\[
	c_0^2=c_1^2\quad\Rightarrow\quad c_1=\pm c_0\,.
\]
From the equations in block $(*d)$ we obtain, after substitution of the
previously found equations,
\begin{align*}
	\frac{\del v^2}{\del s_1} &= -a_{10} \\
	\frac{\del v^3}{\del s_0} &= -a_{01} \\
	\frac{\del v^2}{\del s_0}-\frac{\del v^3}{\del s_1} &= a_{00}-a_{11}
\end{align*}
Next, substituting everything found so far, the block $(*A)$ yields
\[
	a_{01}=0\,,\quad
	a_{11}=-\frac23a_{00}
\]
Finally, with these solutions, we infer from the block $(*D)$ that
\[
	\frac{\del v^2}{\del s_0} = -\frac{10}{3}a_{00}\,,
\]
and we obtain
\begin{align*}
	v^0 &= -\frac53a_{00}x_0+a_{10} \\
	v^1 &= -\frac53a_{00}x_1+a_{10} \\
	v^2 &= -\frac{10}{3}a_{00}s_0-a_{10}s_1+\nu \\
	v^3 &= -\frac53a_{00}s_1+\nu_0
\end{align*}
and thus a 4-dimensional algebra of c-projective symmetries,
$$ \projalg = \langle
	\del_{s_0}\,,\ 
	\del_{s_1}\,,\ 
	x_0\del_{x_0}+x_1\del_{x_1}+s_0\del_{s_0}+s_1\del_{s_1}\,,\ 
	\del_{x_0}+\del_{x_1}-s_1\del_{s_0}	
		\rangle\,. $$

\subsubsection{Metrics L2}
For the metrics of type L2, Condition~\eqref{eqn:L.exceptional.2} becomes
$$ c_1^2d_0^2 = c_0^2d_1^2\,. $$
Solving the remaining equations, we find, analogously to the case L1,
\begin{gather*}
	a_{01}=0\,,\ a_{11}=-\frac23a_{00}\,,\\
	\beta=0\,,
\end{gather*}
and then, solving for the derivatives of $v^j$,
$$ \projalg = \langle
				\del_{s_0}\,,\
				\del_{s_1}\,,\
				\del_{x_0}+\del_{x_1}-2s_0\del_{s_0}-s_1\del_{s_1}\,,\
				\frac{1}{c_0}e^{x_0}\del_{x_0}
					+\frac{1}{c_1}e^{x_1}\del_{x_1}
					+s_1\del_{s_0}
			\rangle\,. $$
Here we have seen that, for the exceptional case of Condition~\eqref{eqn:L.exceptional.2}, the algebra of c-projective symmetries is 4-dimensional and $\beta=0$. We are going to see in Section~\ref{sec:finalising.L2} that, if $\beta=0$, the algebra is 4-dimensional even if~\eqref{eqn:L.exceptional.2} is not satisfied.

\subsubsection{Metrics L3}
Condition~\eqref{eqn:L.exceptional.2} becomes
$$ c_0^2e^{-3x_0} = c_1^2e^{-3x_1}\,, $$
which cannot be identically satisfied unless $(c_0,c_1)=(0,0)$. These values, however, are impossible as the corresponding metric of L3 type would become degenerate.
We conclude that the exceptional branch is void in the case L3.

\subsubsection{Metrics L4} 
Using that $\cos(x_0)e^{3\beta x_0}$ and $\cos(x_1)e^{3\beta x_0}$ are
linearly independent functions, Condition~\eqref{eqn:L.exceptional.2} is 
equivalent to the system
\[
	c_0^2 = c_1^2 = 0\,,
\]
which cannot be satisfied as the corresponding metric of L4 type would be degenerate.
We conclude that the exceptional branch is void in the case L4.

\subsection{Finalising the proof of Theorem~\ref{thm:Liouville.main}}\label{sec:liouville-finalise}
We return to the main branch, i.e.~we now assume
\[
	 F_1^2\rho_0'^2+\varepsilon F_0^2\rho_1'^2\ne0\,.
\]
Using Lemma~\ref{la:Liouville}, we find that the blocks $(*a)$, $(*b)$,
$(*c)$ and $(*d)$ as well as $(*B)$ and $(*C)$ become redundant.
The block $(*A)$ is diagonal and the block $(*D)$ is symmetric, and these
remaining conditions can be written as equations of the form
\[
\sum_\ell c_\ell F_\ell(x_0,x_1) = 0
\]
where $F_\ell(x_0,x_1)$ are bivariate, linearly independent functions in $x_0,x_1$. Thus $c_\ell=0$ for all $\ell$.
We solve these remaining conditions separately for each of the metrics of
Liouville type. Each of the following subsections covers one of these metrics.

\subsubsection{Case L1}

The remaining conditions~\eqref{eqn:liouville.cases.algebraic.condition} are 
rational functions in $x_0,x_1$ and thus lead to polynomial conditions, 
where each coefficient has to vanish independently.
We arrive at
$$
	\mu = -\frac{10}{3}a_{00}\,,\quad
	a_{01}=0\,,\quad
	a_{11} = -\frac23a_{00}\,,
$$
and thus, up to a constant factor,
$$
	v = \begin{pmatrix}
			-5a_{00}x_0+3a_{10} \\
			-5a_{00}x_1+3a_{10} \\
			-10a_{00}s_0-3a_{10}s_1+3\nu \\
			-5a_{00}s_1+3\nu_0
		\end{pmatrix}
$$
leading to the 4-dimensional algebra
$$
	\mathfrak{p}=\langle
					\partial_{s_0}\,,\quad
					\partial_{s_1}\,,\quad
		\partial_{x_0}+\partial_{x_1}-s_1\partial_{s_0}\,,\quad
		x_0\partial_{x_0}+x_1\partial_{x_1}+2s_0\partial_{s_0}+s_1\partial_{s_1}
				\rangle\,.
$$
We note that this algebra coincides with the one found in the exceptional
case, see Section~\ref{sec:exceptional.L1}.

\subsubsection{Case L2}\label{sec:finalising.L2}
Since $\beta\ne1$, the remaining 
conditions~\eqref{eqn:liouville.cases.algebraic.condition} are equivalent to
\begin{align*}
	a_{01}c_0^2e^{2\beta x_0} &= 0 \\
	a_{01}c_1^2e^{2\beta x_1} &= 0 \\
	a_{10}\beta e^{2\beta x_0} &= 0 \\
	a_{10}\beta e^{2\beta x_1} &= 0 \\
	((5\beta-2)a_{00}-3a_{11})c_0 &= 0 \\
	((5\beta-2)a_{00}-3a_{11})c_1 &= 0
\end{align*}
We solve this system distinguishing two branches:\smallskip

\noindent\textbf{Case $\beta\ne0$.}
We arrive at
$$
	\mu = -\frac{5}{3}(\beta+2)a_{00}\,,\
	a_{01}=0\,,\
	a_{10}=0\,,\
	a_{11} = \frac13(5\beta-2)a_{00}\,.
$$
This yields the 3-dimensional algebra
$$
	\mathfrak{p}=\langle
					\partial_{s_0}\,,\quad
					\partial_{s_1}\,,\quad
					\partial_{x_0}+\partial_{x_1}
						-(\beta+2)s_0\partial_{s_0}
						-(2\beta+1)s_1\partial_{s_1}
				\rangle\,.
$$

\noindent\textbf{Case $\beta=0$.}
In this case we arrive at
$$
	\mu = -\frac{10}{3}a_{00}\,,\
	a_{01}=0\,,\
	a_{11} = -\frac23a_{00}\,.
$$
This yields the 4-dimensional algebra
$$
	\mathfrak{p}=\langle
					\partial_{s_0}\,,\quad
					\partial_{s_1}\,,\quad
					\frac1{c_0}e^{x_0}\del_{x_0}
						+\frac1{c_1}e^{x_1}\del_{x_1}
						+s_1\del_{s_0}\,,\quad
					\partial_{x_0}+\partial_{x_1}
						-2s_0\partial_{s_0}
						-s_1\partial_{s_1}
				\rangle\,.
$$

\subsubsection{Case L3}
We obtain the conditions
$$
	\mu = -5a_{00}\,,\quad
	a_{01}=0\,,\quad
	a_{10}=\frac53a_{00}\,,\quad
	a_{11}=a_{00}\,,
$$
and arrive at the 3-dimensional algebra
$$
\mathfrak{p}=\langle
				\partial_{s_0}\,,\quad
				\partial_{s_1}\,,\quad
				\partial_{x_0}+\partial_{x_1}
					-(3s_0+s_1)\partial_{s_0}
					-3s_1\partial_{s_1}
			\rangle\,.
$$

\subsubsection{Case L4}
Solving the remaining conditions 
of~\eqref{eqn:liouville.cases.algebraic.condition},
and using that $\cos(x_0)\sin(x_0)$ and $\cos^2(x_0)$
are linearly independent functions, we find the system
\begin{align*}
	(a_{01}+a_{10})\beta-a_{00}+a_{11} &= 0 \\
	(a_{00}-a_{11})\beta+a_{01}+a_{10} &= 0 \\
	3a_{01}\beta-5a_{00} &= 0\,.
\end{align*}
We thus distinguish two cases:\smallskip

\noindent\textbf{Case $\beta\ne0$.}
We arrive at
$$
	\mu = -5a_{00}
	a_{01}=\frac53\frac{a_{00}}{\beta},
	a_{10}=-\frac53\frac{a_{00}}{\beta},
	a_{11}=a_{00}
$$
and thus the 3-dimensional algebra
\begin{equation}\label{eq:projalg.finalising.L4.beta.ne.0}
	\mathfrak{p}=\langle
					\partial_{s_0}\,,\quad
					\partial_{s_1}\,,\quad
					\partial_{x_0}+\partial_{x_1}
						-(3\beta s_0-s_1)\partial_{s_0}
						-(s_0+3\beta s_1)\partial_{s_1}
				\rangle\,.
\end{equation}

\noindent\textbf{Case $\beta=0$.}
The remaining equation of (2) can be solved for
$$ a_{00}=0\,,\quad a_{11}=0\,,\quad a_{01}=-a_{10} $$
and the remaining conditions of (3) then yield the 3-dimensional algebra
$$ \projalg = \langle
				\del_{s_0}\,,\quad
				\del_{s_1}\,,\quad
				\del_{x_0}+\del_{x_1}+s_1\del_{s_0}-s_0\del_{s_1}
			\rangle\,. $$
We observe that this algebra is analogous to~\eqref{eq:projalg.finalising.L4.beta.ne.0}, which holds for $\beta=0$.

\section{Proof of Theorem~\ref{thm:complex.main} (complex
type)}\label{sec:complex}
Theorem~\ref{thm:complex.main} is a direct corollary of Theorem~\ref{thm:Liouville.main}.
Performing the transformation
$$
	s_0^\text{new}=4s_0^\text{old}\,,\qquad
	s_1^\text{new}=4s_1^\text{old}\,,
$$
the metric~\eqref{eqn:C.metrics} becomes
\begin{align*}
	g &= -\frac14\,\Bigg(
		(\rho-\bar\rho)(F^2dz^2-\bar F^2d\barz^2)
		+\frac{1}{\rho-\bar\rho}\left[
		\left(\frac{\rho'}{F}\right)^2(ds_0+\bar\rho ds_1)^2
		-\left(\frac{\bar\rho'}{\bar F}\right)^2(ds_0+\rho ds_1)^2
		\right]
	\Bigg)
\end{align*}
Up to the irrelevant, constant conformal factor of $-\frac14$, we observe a formal analogy of the complex type metrics and the Liouville type metrics where we formally replace
$$ \veps=-1\,, \quad x_0\to z\,,\quad x_1\to\barz\text{ etc.}, $$
as well as $c_0\to\varsigma$ and $c_1\to\bar\varsigma$ (cases C1, C3, C4) or $c_0,c_1\to1$, $d_0\to\varsigma$, $d_1\to\bar\varsigma$ (case C2).
The statement then follows directly by comparison, in an entirely analogous manner.

\section{Proof of Theorem~\ref{thm:degenerate.main} (Degenerate type)}\label{sec:degenerate}

We proceed again by proving an auxiliary result first. It provides us with a partial solution for the c-projective vector fields, allowing us to lift 2-dimensional homotheties to the K\"ahler surface if certain integration conditions hold.
\begin{proposition}\label{prop:degenerate}~
	
	\noindent
	(i) Any c-projective vector field of a degenerate metric from
	Fact~\ref{fct:metrics} is of the form
	\[
		v = v^0(x_0)\partial_{x_0}+(\eta x_1+f(s_0,s_1))\partial_{x_1}+u
	\]
	where $u$ is a homothetic vector field of $h$, $\eta\in\RR$.
	
	\noindent
	(ii) Let $g$ be a metric of type D1--D3 of Fact~\ref{fct:metrics} and let
	$a_{00},a_{01},a_{10}$ and $a_{11}$ be constants such that
	\begin{equation}\label{eqn:choice.aij.Dx}
	\begin{aligned}
		\text{D1:}&\quad a_{10}=a_{11}\text{~and~}a_{01}=a_{00}
		\\
		\text{D2a:}&\quad a_{10}=a_{11}-a_{00}\text{~and~}a_{01}=0
		\\
		\text{D2b:}&\quad a_{10}=(\beta-1)a_{00}\text{~and~}a_{01}=0
		\text{~and~}a_{11}=\beta a_{00}
		\\
		\text{D3:}&\quad a_{10}=\tfrac12(a_{11}-a_{00})\text{~and~}
				a_{01}=-\tfrac12(a_{11}-a_{00})\qquad
	\end{aligned}
	\end{equation}
	hold.
	Furthermore, let $u=u^0\del_{s_0}+u^1\del_{s_1}$ be a homothetic vector field of $h$, $\lie_uh=Ch$, that satisfies the respective integrability condition in the last column of Table~\ref{tab:integrability.Dx}.
	Finally, let $v^0(x_0)$, $f(s_0,s_1)$ and $\eta\in\RR$ be as specified in the middle column of Table~\ref{tab:integrability.Dx}.
	Then
	\begin{equation}\label{eqn:v.ansatz.D.metrics}
		v = v^0(x_0)\partial_{x_0}+(\eta x_1+f(s_0,s_1))\partial_{x_1}+u
	\end{equation}
	is a c-projective vector field for the respective metric $g$ of type D1--D3 of
	Fact~\ref{fct:metrics}.
	
	\noindent
	(iii) For a metric of type D1--D3 of Fact~\ref{fct:metrics}, any c-projective vector field arises as in (ii).
\end{proposition}
Note that according to part~(ii) of the proposition, any homothetic vector field of the 2-dimensional metric $h$ can be extended to a c-projective vector field of $g$ in the cases D1, D2b and D3. However, in case of metrics $g$ of type D2a, only Killing vector fields of $h$ can be extended to c-projective vector fields of $g$.
\begin{example}
	Let $g$ be a metric of type D1--D3 of Fact~\ref{fct:metrics}.
	Then
	\[
		v = v^0(x_0)\partial_{x_0}+c\partial_{x_1}\,,\qquad c\in\RR,
	\]
	with
	\begin{align*}
		v^0 &= k\in\RR &&\text{if $g$ is of type D1 or D2a}\\
		v^0 &= 0 &&\text{if $g$ is of type D2b or D3}\,,
	\end{align*}
	are c-projective vector fields of~$g$.
	Indeed, this claim follows immediately from Proposition~\ref{prop:degenerate}, integrating the systems in Table~\ref{tab:integrability.Dx} for $u=0$, i.e.~$f=c\in\RR$, and evaluating the formulas for $v^0$ and $v^1$ in Table~\ref{tab:integrability.Dx} for $C=0$.
\end{example}
\begin{table}
	\textbf{Integrability conditions for part~(ii) of 
	Proposition~\ref{prop:degenerate}}\medskip
	
	\begin{tabular}{l|l|l}
		Case & equations for $v^0$ and $v^1$ & integrability condition
		\\
		\toprule
		D1 &
		\begin{minipage}{.4\textwidth}
			$ v^0 = Cx_0+k\qquad(k\in\RR) $\\[.2cm]
			$ \eta = C $ \\[.2cm]
			$ \diff{f}{s_0} = s_0\,\diff{u^1}{s_0} $ \\
			$ \diff{f}{s_1} =
				u^0+u^1\,s_0\tfrac{G'}{G}-Cs_0 $
		\end{minipage}
		& No condition
		\\
		\midrule
		D2a &
		\begin{minipage}{.3\textwidth}
			$ v^0 = k\qquad(k\in\RR) $\\[.2cm]
			$ \eta = 0 $\\[.2cm]
			$ \diff{f}{s_0} = s_0\diff{u^1}{s_0} $ \\
			$ \diff{f}{s_1} = u^0+\tfrac12s_0\,u^1\,\frac{G'}{G} $
		\end{minipage}
		& $C=0$
		\\
		\midrule
		D2b &
		\begin{minipage}{.4\textwidth}
			$ v^0 = -\frac{C}{\beta+2} $\\[.2cm]
			$ \eta = C $ \\[.2cm]
			$ \diff{f}{s_0} = -\frac1{\beta+2}G\,\diff{u^1}{s_0}
				e^{-(\beta+2)s_0} $ \\
			$ \diff{f}{s_1} = \frac{
					CG
					+(\beta+2)G\,u^0
					-G'\,u^1
				}{
					2(\beta+2)\,e^{(\beta+2)s_0}
				}$
		\end{minipage}
		& No condition
		\\
		\midrule
		D3 &
		\begin{minipage}{.4\textwidth}
			$ v^0 = -\tfrac{C}{9}\,c_1e^{-3x_0} $\\[.2cm]
			$ \eta = C $ \\[.2cm]
			$ \diff{f}{s_0} = -\frac13Ge^{-3s_0}\,\diff{u^1}{s_0} $ \\
			$ \diff{f}{s_1} = \frac{
					CG + 3G\,u^0 - G'\,u^1
				}{
					6\,e^{3s_0}
				}$
		\end{minipage}
		& No condition
	\end{tabular}\smallskip

	\caption{Integrability conditions for lifting homothetic vector fields of the 2-dimensional metric $h$, $\lie_uh=Ch$, to c-projective vector fields of 
	metrics of degenerate type of Fact~\ref{fct:metrics}.}
	\label{tab:integrability.Dx}
\end{table}

Proposition~\ref{prop:degenerate} is proved in the following subsection. In Section~\ref{sec:2D.component.degenerate.metrics} we discuss the homothetic algebra of the 2-dimensional metric $h$. The proof of Theorem~\ref{thm:degenerate.main} is then completed in Section~\ref{sec:degenerate-finalise}.

\subsection{Proof of Proposition~\ref{prop:degenerate}}

\subsubsection{Part~(i)}
The c-projective vector field has the coordinate form
$$ v = v^0(x_0,x_1,s_0,s_1)\del_{x_0}+v^1(x_0,x_1,s_0,s_1)\del_{x_1}+v^2(x_0,x_1,s_0,s_1)\del_{s_0}+v^3(x_0,x_1,s_0,s_1)\del_{s_1}. $$
Our aim now is to obtain necessary conditions for the vector field with components $v^0,v^1,v^2,v^3$ to be a c-projective vector field.
We begin with the first component, i.e.~$v^0$.
The $(0,0)$-component of the block $(*a)$ immediately implies the formula
\begin{equation}\label{eqn:v0.explicit}
	v^0 = -\frac{
		a_{01}\rho^2 + (a_{00}-2a_{01}-a_{11})\rho
		-a_{00}+a_{01}-a_{10}+a_{11}
	}{
		\frac{\partial\rho}{\partial x_0}
	}\,,
\end{equation}
which in particular yields $v^0=v^0(x_0)$ for metrics of type D1--D3.\medskip

Next, we observe from $(*a)$, specifically its $(1,0)$ and $(1,1)$-components, that
\[
	\frac{\partial v^3}{\partial x_0} = \frac{\partial v^3}{\partial x_1} = 0\,.
\]
This exhausts the information in $(*a)$.
Writing out the blocks $(*b)$ and $(*c)$, specifically the $(2,0)$ and the $(2,1)$ component of~\eqref{eqn:LvL}, we furthermore obtain the equations
\[
	\frac{\partial v^2}{\partial x_0} = 0\,,\quad
	\frac{\partial v^2}{\partial x_1} = 0\,.
\]
We therefore conclude that the following preliminary result holds:
\emph{The components $v^2$ and $v^3$ of the c-projective vector field depend on $s_0$ and $s_1$ only, i.e.
$$ v^2=v^2(s_0,s_1)\quad\text{and}\quad v^3=v^3(s_0,s_1)\,. $$
}

\noindent Before analysing these components further, we first consider the component $v^1$ of the c-projective vector field.
Recalling that we consider metrics of type D1--D3 in Fact~\ref{fct:metrics}, we have that $\tau=f_\tau\,ds_1$ with certain functions $f_\tau=f_\tau(s_0,s_1)$ specified in each case D1--D3.
Considering the $(0,3)$-component of~\eqref{eqn:Lvg}, in the $(*B)$ block, we conclude
$$ \diff{v^1}{x_0} = 0, $$
and next, writing out the $(1,2)$-component of~\eqref{eqn:LvL}, in block $(*b)$,
\[
	\frac{\partial v^1}{\partial s_0}
	= f_\tau(s_0,s_1)\,\frac{\partial v^3}{\partial s_0}\,.
\]
Now consider the block $(*d)$, specifically the $(2,2)$ or $(3,3)$ component of~\eqref{eqn:LvL}. We obtain the relation
\[
	a_{11}-a_{00}+a_{01}-a_{10} = 0\,.
\]
The information in~\eqref{eqn:LvL} is now almost entirely exhausted.
The last condition is drawn from block $(*b)$ of Equation~\eqref{eqn:LvL}, specifically its $(1,3)$-component. We infer
\[
	\diff{v^1}{s_1}+f_\tau\diff{v^1}{x_1} = v^2\diff{f_\tau}{s_0}
											+ v^3\diff{f_\tau}{s_1}
											+ f_\tau\diff{v^3}{s_1}\,.
\]
It remains to consider the further conditions arising from~\eqref{eqn:Lvg}. We are going to do this individually for the specific normal forms D1--D3.\smallskip

\noindent\textbf{Case D1.}
Taking the block $(*A)$, specifically the $(0,0)$ and $(1,1)$ components of~\eqref{eqn:Lvg}, we infer
\begin{equation}\label{eqn:D1.v1.x1.formula}
	a_{10} = a_{11}
	\quad\text{and}\quad
	\diff{v^1}{x_1} = a_{00} + a_{11}\,.
\end{equation}

\noindent\textbf{Case D2a.}
Analogously, from the block $(*A)$, we infer
\[
	a_{10} = -a_{00}+a_{11}
	\quad\text{and}\quad
	\diff{v^1}{x_1} = a_{00} + \frac12a_{11}.
\]

\noindent\textbf{Case D2b.}
Analogously, from the block $(*A)$, we infer
$$ a_{10}=(\beta-1)a_{00}\quad\text{and}\quad a_{11}=a_{00}\beta $$
as well as
$$ \diff{v^1}{x_1} = a_{00}(\beta+2). $$

\noindent\textbf{Case D3.}
Analogously, from the block $(*A)$, we infer
$$ a_{10}=\frac12(a_{11}-a_{00})
	\quad\text{and}\quad
	\diff{v^1}{x_1} = \frac32\,a_{00}. $$

\noindent We thus conclude:
\begin{lemma}\label{la:v1.component.degenerate}
	The component $v^1$ is a function $v^1=v^1(x_1,s_0,s_1)$ with $\diff{v^1}{x_1}=:\eta=\text{constant}$.
	Thus
	$$ v^1=\eta x_1+f(s_0,s_1). $$
\end{lemma}

\noindent It remains to show that $u=v^2(s_0,s_1)\del_{s_0}+v^3(s_0,s_1)\del_{s_1}$ satisfies the conditions of a local homothetic vector field of the metric
\begin{equation}\label{eqn:h.f0.f1}
	h=f_0(s_0,s_1)\,ds_0^2+f_1(s_0,s_1)\,ds_1^2\,,
\end{equation}
where
\begin{align*}
	f_0(s_0,s_1)&=G(s_1)\,,\quad f_1(s_0,s_1)=\frac1{G(s_1)} &\text{in the cases D1 and D2a}
	\\
	f_0(s_0,s_1)&=f_1(s_0,s_1)=e^{Cs_0}G(s_1) &\text{in the cases D2b and D3}
\end{align*}
To this end, consider the off-diagonal part of $(*D)$, i.e., the $(3,2)$-component of~\eqref{eqn:Lvg}.
It implies the condition
\begin{equation}\label{eqn:v2.v3.mixed}
	f_0(s_0,s_1)\frac{\partial v^2}{\partial s_1}
	+ f_1(s_0,s_1)\frac{\partial v^3}{\partial s_0} = 0\,.
\end{equation}
This is indeed consistent with the equations of a homothetic vector field $u=u^0\del_{s_0}+u^1\del_{s_1}$ of~$h$.
To prove full equivalence, we proceed again on a case by case basis.

\noindent\textbf{Case D1}
Resubstituting~\eqref{eqn:D1.v1.x1.formula} into the block $(*D)$, we find
\begin{subequations}
	\begin{align}
		\left(a_{00}+a_{11}-2\diff{v^2}{s_0}\right)G-v^3\diff{G}{s_1} &= 0
		\label{D1.(D).1} \\
		(a_{00}+a_{11})G^2
		+v^3\diff{G}{s_1}
		-2G\diff{v^3}{s_1}
		&=0\,.
		\label{D1.(D).2}
	\end{align}
\end{subequations}
Solve~\eqref{D1.(D).1} for $\diff{v^2}{s_0}$,
\begin{equation}\label{eqn:D1.v2.s0}
	\diff{v^2}{s_0} = \frac12\,\left(
			a_{00}+a_{11}-v^3\frac1G \diff{G}{s_1}
		\right)\,.
\end{equation}
and~\eqref{D1.(D).2} for $\diff{v^3}{s_1}$,
\begin{equation}\label{eqn:D1.v3.s1}
	\diff{v^3}{s_1}
	= \frac12\left(
			a_{00}+a_{11}+v^3\frac1G\diff{G}{s_1}
		\right)\,.
\end{equation}
In addition, \eqref{eqn:v2.v3.mixed} becomes
$$ G^2\diff{v^2}{s_1}+\diff{v^3}{s_0} = 0\,. $$
These coincide with the conditions that $u=u^0\del_{s_0}+u^1\del_{s_1}$ is a homothetic vector field of $h$, $\lie_uh=Ch$,
\begin{align*}
	\diff{u^0}{s_0} &= \frac12\,\left( C-u^1\frac1G \diff{G}{s_1} \right)
	\\
	G^2\diff{u^0}{s_1}+\diff{u^1}{s_0} &= 0
	\\
	\diff{u^1}{s_1} &= \frac12\,\left( C+u^1\frac1G\diff{G}{s_1} \right)
\end{align*}
after identifying $u^0=v^2, u^1=v^3$ and $C=a_{00}+a_{11}$.

It remains to consider the block $(*B)$, which is identical to $(*C)$, and
automatically satisfied because of $a_{10} = a_{11}$.
We summarise, with $\eta:=a_{00}+a_{11}$
\begin{align*}
	v^0 &= \eta x_0-a_{00} \\
	v^1 &= \eta x_1+f(s_0,s_1)
\end{align*}
where
\begin{align*}
	\diff{f}{s_0} &= s_0\diff{v^3}{s_0} \\
	\diff{f}{s_1} &= s_0\diff{v^3}{s_1}-\eta s_0+v^2\,.
\end{align*}
\smallskip

\noindent\textbf{Case D2a}
Analogously to case D1, we obtain
\begin{align*}
	\left(2a_{00}+a_{11}-2\diff{v^2}{s_0}\right)G-v^3\diff{G}{s_1} &= 0
	\\
	\left(2a_{00}+a_{11}-2\diff{v^3}{s_1}\right)G
	+v^3\diff{G}{s_1} &= 0\,.
\end{align*}
and then, solving for $\diff{v^2}{s_0}$ and $\diff{v^3}{s_1}$,
\begin{align*}
	\diff{v^2}{s_0} &= \frac12\left(
		2a_{00}+a_{11}-v^3\frac1G \diff{G}{s_1}
	\right)
	\\
	\diff{v^3}{s_1}
	&= \frac12\left( 2a_{00}+a_{11}+v^3\frac1G \diff{G}{s_1} \right)\,.
\end{align*}
Hence $u=v^2\del_{s_0}+v^3\del_{s_1}$ satisfies the conditions of a homothetic vector field of $h$, $\lie_uh=Ch$, with $C=2a_{00}+a_{11}$.
We now verify that the blocks $(*B)$ and $(*C)$ are redundant and do not yield new conditions.
We summarise, with $\eta=a_{00}+\frac12a_{11}$
\[
	v^0 = \frac13(a_{11}-a_{00})
	\quad\text{and}\quad
	v^1 = \eta x_1+f(s_0,s_1)
\]
where
\begin{align*}
	\diff{f}{s_0} &= s_0\diff{v^3}{s_0} \\
	\diff{f}{s_1} &= s_0\diff{v^3}{s_1}-\eta s_0+v^2.
\end{align*}
\smallskip

\noindent\textbf{Case D2b}
Applying a similar reasoning, we arrive at 
\begin{align*}
	\diff{v^2}{s_0} &= \frac12\left(
			(\beta+2)(a_{00}+v^2)-v^3\frac1G\diff{G}{s_1}
		\right)
	\\
	\diff{v^2}{s_1}+\diff{v^3}{s_0} &= 0
	\\
	\diff{v^3}{s_1} &= \frac12\left(
			(\beta+2)(a_{00}+v^2)-v^3\frac1G\diff{G}{s_1}
		\right)
\end{align*}
implying that $u=v^2\del_{s_0}+v^3\del_{s_1}$ satisfies the conditions of a homothetic vector field of $h$, $\lie_uh=Ch$, identifying $C=a_{00}(\beta+2)$.
We then find
\begin{align*}
	v^0 &= -a_{00} \\
	v^1 &= a_{00}(\beta+2)x_1+f(s_0,s_1)
\end{align*}
with
\begin{align*}
	\diff{f}{s_0} &= s_0\diff{v^3}{s_0} \\
	\diff{f}{s_1} &= s_0\diff{v^3}{s_1}-a_{00}(\beta+2)s_0+v^2,
\end{align*}
completing the proof.
\smallskip

\noindent\textbf{Case D3}
Analogously, we arrive at
\begin{align*}
	\diff{v^2}{s_0} &= \frac12\left(
			\frac32(a_{00}+a_{11})+3v^2-v^3\frac1G\diff{G}{s_1}
		\right)
	\\
	\diff{v^3}{s_0}+\diff{v^2}{s_1} &= 0
	\\
	\diff{v^3}{s_1} &= \frac12\left(
			\frac32(a_{00}+a_{11})+3v^2-v^3\frac1G\diff{G}{s_1}
		\right)
\end{align*}
implying that $u=v^2\del_{s_0}+v^3\del_{s_1}$ satisfies the conditions of a homothetic vector field of $h$, $\lie_uh=Ch$, identifying $C=\frac32(a_{00}+a_{11})$.
We then find
\begin{align*}
	v^0 &= -\frac16(a_{00}-a_{11})c_1e^{-3x_0} \\
	v^1 &= \frac32a_{00}x_1+f(s_0,s_1)
\end{align*}
with
\begin{align*}
	\diff{f}{s_0} &= -\frac13Ge^{-3s_0}\diff{v^3}{s_0} \\
	\diff{f}{s_1} &=\frac1{12}\,\left(
	3(a_{00}-a_{11}+4v^2)G-2v^3\diff{G}{s_1}
	\right)e^{-3s_0},
\end{align*}
completing the proof.
\medskip

Together with Equation~\eqref{eqn:v0.explicit} and Lemma~\ref{la:v1.component.degenerate}, the claim of part~(i) follows.

\subsubsection{Part~(ii)}
Part (i) provides necessary conditions that the components of a c-projective vector field have to satisfy.
We therefore have that $v$ has to have the form
\[
	v = v^0(x_0)\partial_{x_0}+(\eta x_1+f(s_0,s_1))\partial_{x_1}+u
\]
with $v^0$ given by~\eqref{eqn:v0.explicit}, where moreover $u=u^0(s_0,s_1)\del_{s_0}+u^1(s_0,s_1)\del_{s_1}$
is a homothetic vector field of the metric $h$ (cf.~\eqref{eqn:h.f0.f1}).
We then choose $a_{00}, a_{01}, a_{10}$ and $a_{11}$ such that~\eqref{eqn:choice.aij.Dx} is satisfied.
Comparing to the conditions of a c-projective vector field, i.e.~\eqref{eqn:LvL} and~\eqref{eqn:Lvg}, determined as in part~(i), we thus find
$$ \eta=C\,, $$
and, moreover, that the vector field $v$ is indeed a c-projective vector field for a metric $g$ of a type D1--D3, if
\begin{subequations}\label{eqn:prolongation.f}
\begin{align}
	\label{eqn:prolongation.f.1}
	\diff{f}{s_0} &= f_\tau(s_0,s_1)\diff{u^1}{s_0}
	\\
	\label{eqn:prolongation.f.2}
	\diff{f}{s_1} &= u^0(s_0,s_1)\diff{f_\tau}{s_0}+u^1(s_0,s_1)+f_\tau(s_0,s_1)-Cf_\tau(s_0,s_1)\,,
\end{align}
\end{subequations}
for which the integrability condition
$$ \frac{\del^2f}{\del s_0\del s_1}=\frac{\del^2f}{\del s_1\del s_0} $$
has to be satisfied.
Equations~\eqref{eqn:prolongation.f} are obtained analogously to the computation in (i); note that the expressions on the right hand sides of~\eqref{eqn:prolongation.f} depend on $s_0,s_1$, the homothetic vector field $u$ of $h$, and the type D1--D3 of the metric $g$ only.
The integrability condition for $f$ is satisfied in the cases D1, D2b and D3. Yet, the integrability condition is not automatically satisfied in the case D2a; in this case, we can integrate for $f$ if and only if
$$ C=0, $$
i.e.~if $u$ is Killing for $h$, $\lie_uh=0$.
As a result, we obtain the formulas and conditions of Table~\ref{tab:integrability.Dx}. In particular, the integrability condition for $f$ is specified in the second column of Table~\ref{tab:integrability.Dx}.
This establishes~\eqref{eqn:LvL} and~\eqref{eqn:Lvg} for the ansatz~\eqref{eqn:v.ansatz.D.metrics} for $v$. It can be shown that this implies that $v$ is a c-projective vector field. We conclude the proof of part~(ii) of Proposition~\ref{prop:degenerate} by explicitly confirming that $v$ is c-projective, proceeding case by case using Lie symmetry methods analogous to those outlined in Appendix~\ref{app:fubini-study}.

\subsubsection{Part~(iii)}
The third part of the claim is easily verified by a careful inspection of the proofs of parts~(i) and~(ii). Indeed, in the considerations for part~(i) we have exhausted all conditions for a c-projective vector field of a metric $g$ among the types D1--D3.
In part (ii) we have seen that given a proper choice of a homothetic vector field of $h$ together with suitable constants $a_{00}$ to $a_{11}$, the ansatz suggested by~(i) yields a c-projective vector field of $g$. Exhausting the possibilities for $u$ and $a_{ij}$, we therefore obtain all c-projective vector fields of a given $g$ of type D1--D3.

\subsection{Homothetic symmetries for the 2-dimensional component}
\label{sec:2D.component.degenerate.metrics}
Proposition~\ref{prop:degenerate} requires homothetic vector fields of 2-dimensional metrics.
These are then extended to the c-projective vector fields of $g$ we are seeking.
In the current section we are therefore going to investigate the homothetic algebras of certain real surfaces.
According to Fact~\ref{fct:metrics}, we need to consider the metrics
\begin{equation}\label{eqn:metrics.f1f}
	G(s_1)ds_0^2+\frac{1}{G(s_1)}ds_1^2
\end{equation}
and
\begin{equation}\label{eqn:metrics.f}
	e^{Cs_0}G(s_1)(ds_0^2+ds_1^2)\,.
\end{equation}
Without loss of generality, we may suppose $G>0$ for our purposes here, as we work locally and a constant conformal factor is irrelevant for the question under investigation. In particular, the homothetic algebra of the 2-dimensional metric $h$ remains unaffected by a change of the sign of $G$.
Note, however, that the sign of $G$ becomes relevant when computing the actual c-projective vector fields using Proposition~\ref{prop:degenerate}, c.f.~also Table~\ref{tab:integrability.Dx}; specifically, the sign of the function $f$ depends on the sign of $G$ in the cases D2b and D3.
For the sake of legibility we are going to suppress some absolute values in some real-valued formulas that would otherwise be ill-defined, particularly in Lemmas~\ref{la:non-generic.algebra.f1f} and~\ref{la:D2b.D3.2D-hom-algebra}, leaving them to the discretion of the reader. Due comments will be made where necessary.

Metrics \eqref{eqn:metrics.f1f} and \eqref{eqn:metrics.f} admit the
obvious homothetic vector field $\partial_{s_0}$, which is Killing for
\eqref{eqn:metrics.f1f}. For \eqref{eqn:metrics.f} $\del_{s_0}$ is Killing if
$C=0$ and properly homothetic otherwise.
Our goal is to identify those metrics of type \eqref{eqn:metrics.f1f}
or~\eqref{eqn:metrics.f} that admit at least one homothetic vector field in
addition to~$\partial_{s_0}$.
We recall that, if a 2-dimensional metric has a 2-dimensional Killing
algebra, it is already of constant curvature (this follows
from~\cite{BMM2008}).

\subsubsection{Metrics \eqref{eqn:metrics.f1f}}
We begin by identifying metrics~\eqref{eqn:metrics.f1f} of 
constant curvature.
The following lemma is obtained by a direct, if cumbersome, computation.
\begin{lemma}\label{la:metric.f1f.constant}
	A metric of type \eqref{eqn:metrics.f1f} has constant curvature if and
	only if
	$$
	\frac{\partial^3G}{\partial s_1^3}=0\,.
	$$
	
	\noindent(i)
	If $G(s_1)=\kappa s_1^2+\mu_1s_1+\mu_2$ with $\kappa\ne0$ then
	\eqref{eqn:metrics.f1f} has non-zero constant curvature.
	We distinguish three subcases according to the sign of the discriminant
	$$ \Delta = \mu_1^2-4\kappa\mu_2 $$
	of $G$ with respect to $s_1$. Thus, a homothetic vector field\footnote{Note that in this particular case, the homothetic algebra coincides with the Killing algebra.} of $h$ is given by
	\begin{align*}
		\Delta&=0 &
		u &= \left(
				\xi_0 +\xi_1\,\left(
					-\kappa s_0^2+4\tfrac{\kappa}{(\diff{G}{s_1})^2}
				\right)
		+\xi_2\,(-2s_0\kappa)
		\right)\del_{s_0}
		+\left(
			\xi_1 s_0\diff{G}{s_1}
			+\xi_2\diff{G}{s_1}
		\right)\del_{s_1}
		\\
		\Delta&>0 &
		u &= \bigg(
		\xi_0 +\xi_1\,\frac{\diff{G}{s_1}}{\sqrt{G}\sqrt{\Delta}}
		\cos(\tfrac12\sqrt{\Delta}s_0)
		-\xi_2\frac{\diff{G}{s_1}}{\sqrt{G}\sqrt{\Delta}}
		\sin(\tfrac12\sqrt{\Delta}s_0)
		\bigg)\del_{s_0}
		\\
		&&&\quad
		+\bigg(
		\xi_1\sin(\tfrac12\sqrt{\Delta}s_0)\sqrt{G}
		+\xi_2\cos(\tfrac12\sqrt{\Delta}s_0)\sqrt{G}
		\bigg)\del_{s_1}
		\\
		\Delta&<0 &
		u &= \bigg(
		\xi_0-\xi_1\,\frac{\diff{G}{s_1}}{\sqrt{G}\sqrt{-\Delta}}
		\exp(\tfrac12\sqrt{-\Delta}s_0)
		+\xi_2\,\frac{\diff{G}{s_1}}{\sqrt{G}\sqrt{-\Delta}}
		\exp(-\tfrac12\sqrt{-\Delta}s_0)
		\bigg)\del_{s_0}
		\\
		&&&\quad
		+\bigg(
		\xi_1\exp(\tfrac12\sqrt{-\Delta}s_0)\sqrt{G}
		+\xi_2\exp(-\tfrac12\sqrt{-\Delta}s_0)\sqrt{G}
		\bigg)\del_{s_1}
	\end{align*}
	where $\xi_i\in\RR$ are real-valued parameters.

	\noindent(ii)
	If $G(s_1)=\mu_1s_1+\mu_2$ with $\mu_1\ne0$, then the metric is flat. Its homothetic algebra
	is 4-dimensional and generated by the properly homothetic vector field
	$$ (\mu_1s_1+\mu_2)\partial_{s_1} $$	
	and the three Killing vector fields
	\begin{align*}
		\del_{s_0}\,,\qquad
		&\frac{\sin(\tfrac12\mu_1s_0)}{\sqrt{\mu_1s_1+\mu_2}}\,\,\del_{s_0}
		-\cos(\tfrac12\mu_1s_0)\sqrt{\mu_1s_1+\mu_2}\,\,\del_{s_1}\,,
		\\
		&\frac{\cos(\tfrac12\mu_1s_0)}{\sqrt{\mu_1s_1+\mu_2}}\,\,\del_{s_0}
		+\sin(\tfrac12\mu_1s_0)\sqrt{\mu_1s_1+\mu_2}\,\,\del_{s_1}\,.
	\end{align*}
	
	\noindent(iii)
	If $G(s_1)=\mu_2\ne0$, then the metric is flat. Its homothetic algebra is 4-dimensional and generated by the properly homothetic vector field
	$$ s_0\,\del_{s_0}+s_1\,\del_{s_1} $$
	and the three Killing vector fields
	\[
		\del_{s_0}\,,\quad
		\del_{s_1}\,,\quad\text{and}\quad
		\frac{s_1}{\mu_2}\,\del_{s_0}-\mu_2s_0\del_{s_1}\,.
	\]
\end{lemma}
Next, we turn to metrics of non-constant Gau{\ss} curvature, aiming to find
those metrics~\eqref{eqn:metrics.f1f} that admit additional homothetic vector
fields, particularly such that are not Killing. Indeed, the existence of two
linearly independent Killing vector fields would already imply that the
metric is of constant curvature \cite{BMM2008}.
The question to be considered is for which $G=G(s_1)$ the
metric~\eqref{eqn:metrics.f1f} admits a vector field $u$ such that
\begin{equation}\label{eqn:homothetic.1}
	\mathcal{L}_u\left(Gds_0^2\pm\frac{1}{G}ds_1^2\right)
	=Gds_0^2\pm\frac{1}{G}ds_1^2\,.
\end{equation}
The following lemma provides an answer. As announced earlier, we are going to omit absolute values for the sake of conciseness and legibility. Indeed, the correct formula in Lemma~\ref{la:non-generic.algebra.f1f} would read
$$ G=k_3\,|k_1s_1+k_2|^{\frac{2(k_1+1)}{k_1}}\,. $$
As the formula would be ill-defined for $k_1s_1+k_2<0$, we may omit these absolute values without any risk of confusion. Note that the homothetic algebra remains unaffected and no solutions are lost as $k_3\in\RR\setminus\{0\}$ (for $k_3=0$ the metric would be ill-defined).
\begin{lemma}\label{la:non-generic.algebra.f1f}
	Let $G(s_1)ds_0^2\pm\frac{1}{G(s_1)}ds_1^2$ be a metric of non-constant curvature
	admitting a homothetic vector field $u$ non-proportional to 
	$\partial_{s_0}$.
	Then
	$$
		G=k_3\,(k_1s_1+k_2)^{\frac{2(k_1+1)}{k_1}}
	$$
	and $u$ is proportional to
	$$
	(k_1+2)s_0\partial_{s_0} - (k_1s_1+k_2)\partial_{s_1}
	$$
	where $k_i\in\mathbb{R}$, $k_1,k_3\ne0$, and $k_1\neq -1$, $k_1\neq -2$ (otherwise the metric is of constant curvature).
\end{lemma}
\begin{proof}
	Let $u=u^0\partial_{s_0}+u^1\partial_{s_1}$.
	Equation~\eqref{eqn:homothetic.1} thus yields the system
	\begin{equation}\label{eqn:system.hom.1}
		\left\{
		\begin{array}{l}
			-2G+u^1G_{s_1}+2Gu^0_{s_0}=0
			\\[.3em]
			G^2u^0_{s_1}\pm u^1_{s_0}=0
			\\[.3em]
			2Gu^1_{s_1}-u^1G_{s_1}-2G=0
		\end{array}
		\right.
	\end{equation}
	Solving the third equation of~\eqref{eqn:system.hom.1},
	\begin{equation}\label{eqn:X2}
		u^1=\sqrt{G}\left( \int\frac{ds_1}{\sqrt{G}}+F(s_0) \right)
	\end{equation}
	where $F=F(s_0)$ is some smooth univariate function.
	Substituting~\eqref{eqn:X2} into~\eqref{eqn:system.hom.1}, we obtain
	\begin{equation}\label{eqn:system.hom.2}
		\left\{
		\begin{array}{l}
			\pm G^2u^0_{s_1}+\sqrt{G}F_{s_0}=0
			\\
			\sqrt{G}G_{s_1}F + \sqrt{G}G_{s_1} \int\frac{ds_1}{\sqrt{G}} +
			2Gu^0_{s_0}-2G=0
		\end{array}
		\right.
	\end{equation}
	Since $G\neq 0$, we can solve this for the derivatives $u^0_{s_1}$ and $u^0_{s_0}$,
	respectively. Using the symmetry of second derivatives,
	$\tfrac{\del^2 u^0}{\del s_1\del s_0}
	-\frac{\del^2 u^0}{\del s_0\del s_1}=0$,
	we arrive at
	\begin{equation*} 
		F(2GG_{s_1s_1}-G_{s_1}^2)
		= \pm 4F_{s_0s_0} - 2\sqrt{G}G_{s_1}
		-2G\int\frac{ds_1}{\sqrt{G}}G_{s_1s_1}
		+\int\frac{ds_1}{\sqrt{G}}G_{s_1}^2
	\end{equation*}
	In this equation, the right-hand side is a sum of a function
	depending only on $s_0$ and a function depending only on $s_1$. Therefore,
	differentiating the left-hand side with respect to $s_0$ and then with
	respect	to $s_1$ yields zero, and we obtain
	$$
	\frac{\partial^2}{\del s_0\del s_1}
		\left(
			F(s_0)\,\left(
					2G\frac{\del^2G}{\del s_1^2}
					-\left(\diff{G}{s_1}\right)^2
				\right)
		\right)=0\,,
	$$
	implying
	\begin{equation*} 
		2G\diff{F}{s_0}\frac{\del^3G}{\del s_1^3}=0\,,
	\end{equation*}
	and thus either $G'''(s_1)=0$ (meaning the metric is of 
	constant curvature due to Lemma~\ref{la:metric.f1f.constant}) or 
	$F'(s_0)=0$.
	Note that, since $G\neq 0$, combining the first equation
	of~\eqref{eqn:system.hom.2} and the second of \eqref{eqn:system.hom.1}
	yields
	$
	u^0_{s_1}=0\,\,
	\Leftrightarrow u^1_{s_0}=0 \,\,
	\Leftrightarrow \,\,F_{s_0}=0\,.
	$
	We have therefore found: If $u$ is a proper homothetic vector field of
	the metric $Gds_0^2\pm\frac{1}{G}ds_1^2$ of non-constant curvature, then
	\begin{equation}\label{eqn:hom.vec.field.usual}
		u=u^0(s_0)\partial_{s_0}+u^1(s_1)\partial_{s_1}\,.
	\end{equation}
	The claim of the lemma is then proven by substituting
	\eqref{eqn:hom.vec.field.usual} into~\eqref{eqn:system.hom.1} and then
	integrating.
\end{proof}

The homothetic algebra in this latter case is 2-dimensional and generated by
\[
	\partial_{s_0}\quad\text{and}\quad
	(k_1+2)s_0\del_{s_0}-(k_1s_1+k_2)\del_{s_1}\,.
\]

\subsubsection{Metrics \eqref{eqn:metrics.f}}
We analyse the metrics~\eqref{eqn:metrics.f} in an analogous manner to the metrics~\eqref{eqn:metrics.f1f}. The following lemma (which is analogous to Lemma~\ref{la:metric.f1f.constant}) is obtained by a direct computation.
\begin{lemma}\label{la:metric.f.constant}~
	
	\noindent(i)
	The metric~\eqref{eqn:metrics.f},
	$$ e^{Cs_0}G(s_1)(ds_0^2+ds_1^2) $$
	with $C\ne0$, has constant curvature if and only
	if
	$$ G(s_1)=k_1e^{k_2s_1} $$
	for constants $k_1,k_2\in\RR$, $k_1\ne0$.
	The constant curvature is then identical to zero and the homothetic
	algebra is 4-dimensional and generated by $\del_{s_0}$ and $\del_{s_1}$
	as well as
	\begin{align*}
		\exp\left(
			-\tfrac{1}{2}Cs_0-\tfrac12 k_2s_1
		\right)
		&\left[
			\sin\left(
				\tfrac12 k_2s_0-\tfrac{1}{2}Cs_1
			\right)\,\del_{s_0}
			-\cos\left(
				\tfrac12 k_2s_0-\tfrac{1}{2}Cs_1
			\right)\,\del_{s_1}
		\right]
		\intertext{and}
		\exp\left(
			-\tfrac{1}{2}Cs_0-\tfrac12 k_2s_1
		\right)
		&\left[
			\cos\left(
				\tfrac12 k_2s_0-\tfrac{1}{2}Cs_1
			\right)\,\del_{s_0}
			+\sin\left(
				\tfrac12 k_2s_0-\tfrac{1}{2}Cs_1
			\right)\,\del_{s_1}
		\right]\,.
	\end{align*}

	\noindent(ii)
	The metric~\eqref{eqn:metrics.f} with $C=0$,
	$$ G(s_1)(ds_0^2+ds_1^2) $$
	has constant curvature if and only if
	$$ G(s_1)=\frac{k_3}{\cos^2\left(k_1s_1+k_2\right)} $$
	for constants $k_1,k_2,k_3\in\RR$, $k_3\ne0$.
	In that case the Gau{\ss} curvature is $-\frac{k_1^2}{k_3}$.
\end{lemma}
\noindent Part~(ii) of Lemma~\ref{la:metric.f.constant} is stated here for completeness and, for conciseness, we abstain from listing the explicit homothetic vector fields in this case as we shall not need them. Note that we will need part~(i) exclusively, for the cases D2b and D3 in Sections~ \ref{sec:D2a.results} and~\ref{sec:D3.results}, respectively.

For the case of non-constant curvature we prove the following lemma. A comment similar to that before Lemma~\ref{la:non-generic.algebra.f1f} is appropriate: The correct formula in Lemma~\ref{la:D2b.D3.2D-hom-algebra} would read
$$ G(s_1)=k_3|\sin(k_1s_1+k_2)|^{\tfrac{C-2k_1}{k_1}}\,, $$
and the sign of $\sin(k_1s_1+k_2)$ would appear, accordingly, in the formula for the homothetic vector field. However, this sign would lead to unnecessarily complicated formulas for the c-projective vector fields. Moreover, the sign can be absorbed by a change of the parameter $k_2$, and so we are going to omit the absolute values for better readability and ease of notation.
\begin{lemma}\label{la:D2b.D3.2D-hom-algebra}
	Let $h=e^{C s_0}G(s_1)(ds_0^2+ds_1^2)$, with $C\ne0$, be a metric 
	of non-constant curvature admitting a homothetic vector field $u$ not
	proportional to $\partial_{s_0}$.
	Then
	$$
		G(s_1)=k_3\sin(k_1s_1+k_2)^{\tfrac{C-2k_1}{k_1}}\,,
	$$
	and $u$ is a linear combination of $\partial_{s_0}$ and
	$$
		w = e^{-k_1s_0}\left(
				\cos(k_1s_1+k_2)\partial_{s_0}
				-\sin(k_1s_1+k_2)\partial_{s_1}
			\right)\,,
	$$
	where $k_i\in\mathbb{R}$, $k_1,k_3\ne0$ and $k_1\ne\frac{C}{2}$. In fact, $w$ is Killing for $h$.
\end{lemma}
\begin{proof}
	Since $C\ne0$, $\partial_{s_0}$ is properly homothetic.
	The existence of $u$ as in the hypothesis, implies the existence of a
	Killing vector field. Without loss of generality we therefore assume $u$
	to be Killing.
	The components $u^0,u^1$ of $u=u^0\del_{s_0}+u^1\del_{s_1}$ and the function $G$ then need to satisfy
	the PDE system
	\begin{subequations}\label{eqn:Killing.u.G}
	\begin{align}
		\label{eqn:Killing.u.G.1}
		u^0_{s_0}-u^1_{s_1} &= 0
		\\
		\label{eqn:Killing.u.G.2}
		u^0_{s_1}+u^1_{s_0} &= 0
		\\
		\label{eqn:Killing.u.G.3}
		C\,u^0G+2Gu^1_{s_1}+u^1G_{s_1} &= 0
	\end{align}
	\end{subequations}
	In addition, one can make use of the classification in
	\cite[Theorem~1]{BMM2008}. It is easily seen that only the cases (1a)
	and (2a) of this classification have the properties implied by the
	hypothesis, particularly the commutator of any Killing vector field and
	any homothetic vector fields is proportional to the Killing vector field.
	We conclude that $[u,\del_{s_0}]=k_1u$ has to hold for some
	constant $k_1\ne0$, i.e.
	\begin{equation}\label{eqn:commutator.u.delx}
		u^0_{s_0}=-k_1u^0\,,\qquad u^1_{s_0}=-k_1u^1
	\end{equation}
	The system composed of~\eqref{eqn:Killing.u.G.1},
	\eqref{eqn:Killing.u.G.2} and~\eqref{eqn:commutator.u.delx} can be
	straightforwardly solved,
	$$
		u^0(s_0,s_1) = \xi\,e^{-k_1s_0}\cos(k_1s_1+k_2)\,,\qquad
		u^1(s_0,s_1) = -\xi\,e^{-k_1s_0}\sin(k_1s_1+k_2)\,,
	$$
	where $\xi\in\RR$.
	Resubstituting into~\eqref{eqn:Killing.u.G.3}, we have
	$$
		\diff{\ln(G)}{s_1} = (C-2k_1)\cot(k_1s_1+k_2)
	$$
	and then
	$$ G = k_3\,\exp\left(\frac{C-2k_1}{k_1}\,\ln\sin(k_1s_1+k_2)\right)
		=k_3\,\sin(k_1s_1+k_2)^{\frac{C}{k_1}-2}
	$$
	with $k_3\ne0$ since $G\ne0$.
\end{proof}

\subsection{Finalising the proof of Theorem~\ref{thm:degenerate.main}}\label{sec:degenerate-finalise}
We finalise the proof by computing the explicit c-projective algebras.

\subsubsection{Case D1}\label{sec:D1.results}
The 2D metric is
\[
	h = G(s_1)ds_0^2+\frac{ds_1^2}{G(s_1)}\,.
\]
We have three scenarios:
\begin{enumerate}
	\item Any homothetic vector field of $h$ is a multiple of the Killing
	vector field $\partial_{s_0}$.
	\item The metric has a $2$-dimensional homothetic algebra. In this
	scenario, w.l.o.g.,
	$$ G(s_1)=k_3(k_1 s_1+k_2)^{\frac{2(k_1+1)}{k_1}} $$
	for $k_1,k_3\ne0$ (Lemma~\ref{la:non-generic.algebra.f1f}).
	\item $G'''(s_1)=0$, but $G''(s_1)\ne0$, and the metric $h$ is of
	non-zero constant curvature, see Lemma~\ref{la:metric.f1f.constant}.
\end{enumerate}

\noindent Note that if $G''(s_1)=0$, then the metric $h$ is flat. In this case,
however, $g$ is already of constant HSC. Indeed, recall that $\dim(\homalg(h))\leq3$ in the case under consideration. If the homothetic algebra of
$h$ would be larger, its Gau{\ss} curvature would already be zero, and thus $g$ would have constant HSC due to Theorem~\ref{thm:CHSC}.
\smallskip

\noindent\textbf{Scenario 1 ($\dim(\homalg(h))=1$).}
We have the homothetic vector fields $u=\xi\partial_{s_0}$ for $h$, and thus obtain
$$ v^0(x_0)=-a_{00} $$
where $a_{00}$ is a free parameter. Moreover, after a straightforward integration, we arrive at
$$ v^1(x_1,s_0,s_1)=\xi s_1+\nu $$
where $\nu\in\RR$. Summarising, the c-projective algebra of $g$ is 3-dimensional and generated by
$$
	\partial_{x_0}\,,\qquad
	\partial_{x_1}\,,\qquad\text{and}\qquad
	s_1\partial_{x_1}+\partial_{s_0}\,.
$$
We thus obtain
$$ a_{00} = a_{01} = -a_{10} = -a_{11}\,, $$
and therefore conclude that $\partial_{x_1}$ and $s_1\partial_{x_1}+\partial_{s_0}$ are Killing vector fields, while
$\del_{x_0}$ is essential as $a_{01}\ne0$.
\smallskip

\noindent\textbf{Scenario 2 ($\dim(\homalg(h))=2$).}
The homothetic algebra of $h$ is parametrised by
$$ \xi\left( (k_1+2)s_0\partial_{s_0} - (k_1s_1+k_2)\partial_{s_1} \right) + \xi_0\del_{s_0} $$
where $\xi,\xi_0\in\RR$.
and we hence find
\[
	a_{11}=2\xi-a_{00}\,,
\]
and then
\[
	v^1 = 2\xi x_1-\xi_0s_0+\nu\qquad
	v^0 = 2\xi x_0-a_{00}\,.
\]
Therefore, the c-projective vector fields of $g$ are obtained as
\[
	v = (2\xi x_0-a_{00})\partial_{x_0}
		+(2\xi x_1-\xi_0s_1+\nu)\partial_{x_1}
		+(\xi(k_1+2)s_0+\xi_0)\partial_{s_0}
		-\xi(k_1s_1+k_2)\partial_{s_1}
\]
with parameters $\xi,\xi_0,a_{00},\nu\in\RR$.
Written separately, the associated generators of the c-projective symmetry
algebra are
$$
	2x_0\partial_{x_0}+2x_1\partial_{x_1}
			+(k_1+2)s_0\partial_{s_0}-(k_1s_1+k_2)\partial_{s_1}\,,
$$
$$
	\partial_{x_0}\,,\qquad
	\partial_{x_1}\,,\qquad\text{and}\qquad
	s_1\partial_{x_1}+\partial_{s_0}\,.
$$

\noindent\textbf{Scenario 3 ($\dim(\homalg(h))=3$).}
The metric $h$ admits a 3-dimensional homothetic algebra if and only if
$G(s_1)=\kappa s_1^2+\mu_1s_1+\mu_2$ with $\kappa\ne0$.
We obtain
$$ a_{11}=-a_{00}\,. $$
We proceed according to the subcases obtained in Lemma~\ref{la:metric.f1f.constant}.
For each subcase, we compute the induced c-projective vector fields analogously to the previous scenarios.\smallskip

\noindent\underline{Subcase $\Delta=0$:} Integrating the PDE system, we find
\[
	v^1 = \xi_1\bigg( \tfrac12 s_0^2(2\kappa s_1+\mu_1)-\frac2{2\kappa
	s_1+\mu_1} \bigg)  +\xi_0 s_1  +\nu
\]
and then
\[
	v = -a_{00}\,\del_{x_0}
		+\bigg(
			\xi_1\big( \tfrac12 s_0^2(2\kappa s_1+\mu_1)-\frac2{2\kappa
			s_1+\mu_1} \big)  +\xi_0 s_1  +\nu
		\bigg)\del_{x_1}
		+u\,.
\]
where $u$ is as in Lemma~\ref{la:metric.f1f.constant}.\smallskip

\noindent\underline{Subcase $\Delta>0$:} We find
\begin{align*}
	v^1 = \xi_0\,s_1
			&+\xi_1\,\sqrt{G}\bigg(
					\tfrac2{\sqrt{\Delta}}\cos(\tfrac12\sqrt{\Delta}s_0)
					+s_0\sin(\tfrac12\sqrt{\Delta}s_0)
				\bigg)
			\\
			&+\xi_2\,\sqrt{G}\bigg(
					s_0\cos(\tfrac12\sqrt{\Delta}s_0)
					-\tfrac2{\sqrt{\Delta}}\sin(\tfrac12\sqrt{\Delta}s_0)
				\bigg)
			+\nu
\end{align*}
and thus
\[
	v = -a_{00}\,\del_{x_0} + v^1\del_{x_1} + u\,,
\]
where $u$ is as in Lemma~\ref{la:metric.f1f.constant}.\smallskip

\noindent\underline{Subcase $\Delta<0$:} We find
\[
v^1 = \xi_0\,s_1
		+\xi_1\,\bigg(
				(s_0\Delta-2)\sqrt{\tfrac{G}{-\Delta}}\exp(\tfrac12 
				\sqrt{-\Delta} s_0)
		\bigg)
		+\xi_2\,\bigg(
				(s_0\Delta+2)\sqrt{\tfrac{G}{-\Delta}}\exp(-\tfrac12 
				\sqrt{-\Delta} s_0)
		\bigg)
		+\nu
\]
and then
\[
	v = -a_{00}\,\del_{x_0}
		+v^1\del_{x_1}
		+u\,,
\]
where $u$ is as in Lemma~\ref{la:metric.f1f.constant}.

\subsubsection{Case D2a}\label{sec:D2a.results}
The metric $h$ is as in the case D1 and we continue along the analogous three
cases.
We have four scenarios:
\begin{enumerate}
	\item Any homothetic vector field of $h$ is a multiple of the Killing
	vector field $\partial_{s_0}$.
	\item The metric has a $2$-dimensional homothetic algebra. In this
	scenario, w.l.o.g., $G(s_1)=\kappa(\mu_1 s_1+\mu_2)^{\frac{2(\mu_1+1)}{\mu_1}}$ for
	$\mu_1,\kappa\ne0$.
	\item $G'''(s_1)=0$, but $G''(s_1)\ne0$ and $G''(s_1)\ne-\frac9{d_1^2}$, such that the metric $h$ is of
	non-zero constant curvature (but $g$ is of non-constant HSC).
	\item $G''(s_1)=0$ and the metric $h$ is flat.
\end{enumerate}
\smallskip

\noindent\textbf{Scenario 1 ($\dim(\homalg(h))=1$).}
For generic $G$, the only homothetic vector fields of $h$ are ($\xi\in\RR$)
$$ u = \xi\del_{s_0}\,. $$
The integrability condition for $v^1$ is
$$ a_{11} = -2a_{00}\,, $$
and we find
\[
v = -a_{00}\partial_{x_0}
+(\xi s_1+\nu)\partial_{x_1}
+\xi\del_{s_0}
\]
with parameters $a_{00},\xi,\nu\in\RR$.
\smallskip

\noindent\textbf{Scenario 2 ($\dim(\homalg(h))=2$).}
According to Lemma~\ref{la:non-generic.algebra.f1f},
$$ G=k_3(k_1s_1+k_2)^{\tfrac{2(k_1+1)}{k_1}} $$
the homothetic vector fields of $h$ are parametrised by ($\xi,\xi_0\in\RR$)
$$ u=\xi_0\partial_{s_0}
	+\xi\,\big((k_1+2)s_0\del_{s_0}-(k_1s_1+k_2)\del_{s_1}\big)\,. $$
From Table~\ref{tab:integrability.Dx}, we infer the integrability condition
$$ \xi=0\,, $$
as only Killing vector fields of $h$ can be extended to c-projective vector fields of the metric~$g$.
We find ($a_{00},\xi_0,\nu\in\RR$)
\[
	v = -a_{00}\partial_{x_0}
		+(\xi_0 s_1+\nu)\partial_{x_1}
		+\xi_0\del_{s_0}
\]
and therefore the c-projective algebra of $g$ is generated by
$$ \partial_{x_0}\,,\quad
	s_1\partial_{x_1}+\del_{s_0}\,,\qquad
	\partial_{x_1}\,. $$
\smallskip

\noindent\textbf{Scenario 3 ($\dim(\homalg(h))=3$).}
We turn to the case when $h$ has a 3-dimensional homothetic algebra, which is analogous to Scenario~3 of case D1.
Using Lemma~\ref{la:metric.f1f.constant}, we obtain $a_{11}=-2a_{00}$.\smallskip

\noindent\underline{$\Delta=0$:}
\[
	v^1 = \xi_0\,s_1
			+\xi_1\,\left(
					\frac12s_0^2(2\kappa s_1+\mu_1)
					-\frac2{2\kappa s_1+\mu_1}
					\right)
			+\nu
\]
and then
\[
	v = -3a_{00}\del_{x_0}+v^1\del_{x_1}+u\,.
\]
\noindent\underline{$\Delta>0$:}
\begin{align*}
	v^1 = \xi_0\,s_1
		&+\xi_1\,\bigg(
			2\tfrac{\sqrt{G}}{\sqrt{\Delta}}\cos(\tfrac12\sqrt{\Delta} s_0)
			+\sqrt{G}s_0\sin(\tfrac12 s_0 \sqrt{\Delta})
		\bigg)
	\\
		&+\xi_2\,\bigg(
			\sqrt{G}s_0\cos(\tfrac12 s_0 \sqrt{\Delta})
			-2\tfrac{\sqrt{G}}{\sqrt{\Delta}}\sin(\tfrac12 s_0 \sqrt{\Delta})
		\bigg)
		+\nu
\end{align*}
and then
\[
	v = -3a_{00}\del_{x_0}+v^1\del_{x_1}+u\,.
\]

\noindent\underline{$\Delta<0$:}
\begin{align*}
	v^1 = \xi_0\,s_1
			&+\xi_1\,(s_0+\tfrac2{\sqrt{-\Delta}})
					\sqrt{G}\exp(\tfrac12\sqrt{-\Delta}s_0)
		\\
			&+\xi_2\,(s_0-\tfrac2{\sqrt{-\Delta}})
					\sqrt{G}\exp(-\tfrac12\sqrt{-\Delta}s_0)
		+\nu
\end{align*}
and then
\[
	v = -3a_{00}\del_{x_0}+v^1\del_{x_1}+u\,.
\]
In all three subcases, the vector field $u$ is as in Lemma~\ref{la:metric.f1f.constant}. The parameters are $a_{00},\xi_0,\xi_1,\xi_2,\nu\in\RR$.
\smallskip

\noindent\textbf{Scenario 4 ($\dim(\homalg(h))=4$).}
According to Lemma~\ref{la:metric.f1f.constant},
$$ G=\mu_1s_1+\mu_2\,, $$
and we need to distinguish the cases $\mu_1\ne0$ and $\mu_1=0$.\smallskip

\noindent\underline{Subcase $\mu_1\ne0$:}
The homothetic vector fields of $h$ are parametrised by
\begin{align*}
	u = \xi_0\del_{s_0}
	&+\xi_1 \bigg(
	\frac{\sin(\tfrac12\mu_1s_0)}{\sqrt{\mu_1s_1+\mu_2}}\del_{s_0}
	-\cos(\tfrac12\mu_1s_0)\sqrt{\mu_1s_1+\mu_2}\del_{s_1}
	\bigg)
	\\
	&+\xi_2 \bigg(
	\frac{\cos(\tfrac12\mu_1s_0)}{\sqrt{\mu_1s_1+\mu_2}}\del_{s_0}
	+\sin(\tfrac12\mu_1s_0)\sqrt{\mu_1s_1+\mu_2}\del_{s_1}
	\bigg)
	+\xi\,\del_{s_1}\,,
\end{align*}
where $\xi$ parametrises proper homothetic vector fields and the parameters $\xi_i$ describe the Killing vector fields.
Due to the integrability condition, we have
$$ \xi=0\,. $$
We obtain
\begin{align*}
 v^1 = \xi_0s_1
	&+\frac{\xi_1}{\mu_1}\,(\mu_1s_0\sin(\tfrac12\mu_1s_0)
					+2\cos(\tfrac12\mu_1s_0))\sqrt{\mu_1s_1+\mu_2}
	\\
	&+\frac{\xi_2}{\mu_1}\,(\mu_1s_0\cos(\tfrac12\mu_1s_0)
					-2\sin(\tfrac12\mu_1s_0))\sqrt{\mu_1s_1+\mu_2}
	+\nu
\end{align*}
where $\nu\in\RR$ is an integration constant. Moreover, we have
$$ v^0 = -a_{00} $$
and thus $g$ admits the 5-dimensional c-projective algebra generated by
$$ \del_{x_0}\,,\quad \del_{x_1}\,,\quad s_1\del_{x_1}+\del_{s_0} $$
as well as
\begin{multline*}	
	\big(
		(\mu_1s_0\sin(\tfrac12\mu_1s_0)
		+2\cos(\tfrac12\mu_1s_0))\sqrt{\mu_1s_1+\mu_2}
	\big)\del_{x_1}
	\\
	+\mu_1\,\frac{\cos(\tfrac12\mu_1s_0)}{\sqrt{\mu_1s_1+\mu_2}}\del_{s_0}
	+\mu_1\,\sin(\tfrac12\mu_1s_0)\sqrt{\mu_1s_1+\mu_2}\del_{s_1}
\end{multline*}
and
\begin{multline*}
	 \big(
		(\mu_1s_0\cos(\tfrac12\mu_1s_0)
		-2\sin(\tfrac12\mu_1s_0))\sqrt{\mu_1s_1+\mu_2}
	\big)\del_{x_1}
	\\
	-\mu_1\,\frac{\sin(\tfrac12\mu_1s_0)}{\sqrt{\mu_1s_1+\mu_2}}\del_{s_0}
	+\mu_1\,\cos(\tfrac12\mu_1s_0)\sqrt{\mu_1s_1+\mu_2}\del_{s_1}\,.
\end{multline*}

\noindent\underline{Subcase $\mu_1=0$:}
The homothetic vector fields of $h$ are parametrised by
$$ u = \xi\,(s_0\,\del_{s_0}+s_1\,\del_{s_1})
		+\xi_1\,\left( \frac{s_1}{\mu_2}\,\del_{s_0}-\mu_2s_0\del_{s_1} \right)
		+\xi_2\del_{s_0}+\xi_3\del_{s_1} $$
where $\xi$ again parametrises proper homothetic vector fields; the parameters $\xi_i$ describe the Killing vector fields.
Due to the integrability condition,
$$ \xi=0\,, $$
and we hence obtain
$$ v^1 = \xi_1\left( -\frac{\mu_2}{2}\,s_0^2+\frac1{2\mu_2}\,s_1^2 \right) +\xi_2\,s_1+ \nu $$
($\xi_i,\nu\in\RR$) and thus $g$ admits the 5-dimensional c-projective algebra generated by
$$ \del_{x_0}\,,\quad \del_{x_1}\,, $$
as well as
\[
	s_1\del_{x_1}+\del_{s_0}\,,\qquad
	\del_{s_1}\,,
\]
and
\[
	\left( -\frac{\mu_2}{2}\,s_0^2+\frac1{2\mu_2}\,s_1^2 \right)\del_{x_1}+\frac{s_1}{\mu_2}\,\del_{s_0}-\mu_2s_0\del_{s_1}\,.
\]

\subsubsection{Case D2b}\label{sec:D2b.results}
The metric $h$ is
$$ h = e^{-(\beta+2)s_0}G(s_1)(ds_0^2+ds_1^2)\,. $$
We thus have the following three scenarios:
\begin{enumerate}
	\item Any homothetic vector field of $h$ is a multiple of the proper
	homothetic vector field $\partial_{s_0}$.
	\item The metric $h$ admits a 2-dimensional homothetic algebra and is as
	in Lemma~\ref{la:D2b.D3.2D-hom-algebra} with $C=-(\beta+2)$.
	\item The metric $h$ has constant curvature. Due to Lemma~\ref{la:metric.f.constant}, it follows that $h$ is flat.
\end{enumerate}
\smallskip

\noindent\textbf{Scenario 1 ($\dim(\homalg(h))=1$).}
The metric $h$ in this scenario only admits the proper homothetic vector
fields $u = \xi\del_{s_0}$ ($\xi\in\RR$).
Following the same steps as before, we find $v^0=\xi$ and the PDE system for
$v^1(x_1,s_0,s_1)$,
\[
	\diff{v^1}{x_1}=(\beta+2)\xi\,,\qquad
	\diff{v^1}{s_0}=0\,,\quad
	\diff{v^1}{s_1}=0\,,
\]
implying
\[
	v^1 = \xi(\beta+2)x_1+\nu\,.
\]
We therefore arrive at
\[
	v = \xi(\del_{x_0}+(\beta+2)x_1\del_{x_1}+\del_{s_0})
		+ \nu\del_{x_1}
\]
($\xi,\nu\in\RR$). Thus the c-projective algebra of $g$ is 2-dimensional.
\smallskip

\noindent\textbf{Scenario 2 ($\dim(\homalg(h))=2$).}
If the homothetic algebra is 2-dimensional, we use
Lemma~\ref{la:D2b.D3.2D-hom-algebra} with $C=-(\beta+2)$.
We have
$$ u = \xi\del_{s_0}+\xi_1 w\,, $$
($w$ as in Lemma~\ref{la:D2b.D3.2D-hom-algebra}) and obtain, if $k_1\ne-(\beta+2)$,
$$ v^1 = \xi_1\,\frac{k_1k_3}{(k_1+\beta+2)(\beta+2)}\,
\sin(k_1s_1+k_2)^{-\frac{k_1+\beta+2}{k_1}}\,e^{-s_0(k_1+\beta+2)}
-\xi\,(\beta+2)x_1+\nu $$
where $\nu\in\RR$. If $k_1=-(\beta+2)$, we obtain
\begin{align*}
	v^1 &= \xi_1\,\frac{k_3}{\beta+2}\,\bigg( s_0(\beta+2)-\ln\sin((\beta+2)s_1 -k_2) \bigg)
	-\xi\,(\beta+2)x_1+\nu
\end{align*}
where again $\nu\in\RR$.
The c-projective algebra of $g$ is therefore 3-dimensional. If $k_1\ne-(\beta+2)$, it is generated by
$$ \del_{x_1}\,,\quad
	\del_{x_0}
	-(\beta+2)x_1\del_{x_1}
	+\del_{s_0} $$
and
\begin{multline*}
	\frac{k_1k_3}{(k_1+\beta+2)(\beta+2)}\,
		\sin(k_1s_1+k_2)^{-\frac{k_1+\beta+2}{k_1}}\,
		e^{-s_0(k_1+\beta+2)}\del_{x_1}
	\\
	+ e^{-k_1s_0}\cos(k_1s_1+k_2)\del_{s_0}
	- e^{-k_1s_0}\sin(k_1s_1+k_2)\del_{s_1}\,.
\end{multline*}
If, on the other hand, $k_1=-(\beta+2)$, then it is generated by
$$ \del_{x_1}\,,\quad
	\del_{x_0}
	-(\beta+2)x_1\del_{x_1}
	+\del_{s_0} $$
and
\begin{multline*}
	\frac{k_3}{\beta+2}\big( s_0(\beta+2)-\ln\sin((\beta+2)s_1 -k_2) \big)\,\del_{x_1}
	\\
	+ e^{(\beta+2)s_0}\cos((\beta+2)s_1-k_2)\del_{s_0}
	+ e^{(\beta+2)s_0}\sin((\beta+2)s_1-k_2)\del_{s_1}\,.
\end{multline*}
\smallskip

\noindent\textbf{Scenario 3 ($\dim(\homalg(h))=4$).}
In this case $h$ is flat and, due to Lemma~\ref{la:metric.f.constant},
$$ G = e^{\mu_1s_1} $$
and the homothetic algebra of $h$ is parametrised by
\begin{align*}
	u &= \xi_0\del_{s_0}
		+\xi_1\,(\mu_1\del_{s_0}+(\beta+2)\del_{s_1})
	\\
	&	+\xi_2\,(\exp(\tfrac{s_0}{2}(\beta+2)-\tfrac12\mu_1s_1)\sin(\tfrac12\mu_1s_0+\tfrac12s_1(\beta+2))\del_{s_0}
	\\
	&\qquad		-\exp(\tfrac12s_0(\beta+2)-\tfrac12\mu_1s_1)\cos(\tfrac12\mu_1s_0+\tfrac12s_1(\beta+2))\del_{s_1})
	\\
	&	+\xi_3\,(\exp(\tfrac12s_0(\beta+2)-\tfrac12\mu_1s_1)\cos(\mu_1s_0+\tfrac12s_1(\beta+2))\del_{s_0}
	\\
	&\qquad		+\exp(\tfrac12s_0(\beta+2)-\tfrac12\mu_1s_1)\sin(\tfrac12\mu_1s_0+\tfrac12s_1(\beta+2))\del_{s_1})
\end{align*}
where $\xi_i\in\RR$.
We hence obtain ($\xi_i,\nu\in\RR$)
\begin{align*}
	v^1 &= \nu -\xi_0\,(\beta+2)x_1
	\\
	&\quad +\xi_2\exp(\tfrac12\mu_1s_1-\tfrac12s_0(\beta+2))
			\frac{
				 2\mu_1\sin(W)(\beta+2)
				 -\cos(W)\,((\beta+2)^2-\mu_1^2)
			}{
				((\beta+2)^2+\mu_1^2)(\beta+2)
			}
	\\
	&\quad +\xi_3\exp(-\tfrac12(\beta+2)s_0+\tfrac12\mu_1s_1)
			\frac{
				2\mu_1\cos(W)(\beta+2)
				-\sin(W)(\mu_1^2-(\beta+2)^2)
			}{
				((\beta+2)^2+\mu_1^2)(\beta+2)
			}
\end{align*}
with $W:=\frac12s_1(\beta+2)+\frac12\mu_1s_0$.
Moreover, we obtain the restriction
$$ a_{00}=-\xi_0 $$
and thus find a 5-dimensional c-projective algebra for $g$, parametrised by
\[
	v = -\xi_0\del_{x_0}+v^1\del_{x_1}+u
\]
($\xi_0,\xi_1,\xi_2,\xi_3,\nu\in\RR$).

\subsubsection{Case D3}\label{sec:D3.results}
The metric $h$ is
$$ h = e^{-3s_0}G(s_1)(ds_0^2+ds_1^2)\,. $$
We thus have the following three scenarios:
\begin{enumerate}
	\item Any homothetic vector field of $h$ is a multiple of the proper
	homothetic vector field $\partial_{s_0}$.
	\item The metric $h$ admits a 2-dimensional homothetic algebra and is as
	in Lemma~\ref{la:D2b.D3.2D-hom-algebra} with $C=-3$.
	\item The metric $h$ has constant curvature. Due to Lemma~\ref{la:metric.f.constant}, it follows that $h$ is flat.
\end{enumerate}
\smallskip

\noindent\textbf{Scenario 1 ($\dim(\homalg(h))=1$).}
The metric $h$ in this scenario only admits the proper homothetic vector
fields $u = \xi\del_{s_0}$ ($\xi\in\RR$).
Following the same steps as before, we find $v^0=\xi$ and the PDE system for
$v^1(x_1,s_0,s_1)$,
\[
	\diff{v^1}{x_1}+3\xi=0\,,\qquad
	\diff{v^1}{s_0}=0\,,\quad
	\diff{v^1}{s_1}=0\,,
\]
and therefore we arrive at ($\xi,\nu\in\RR$)
\[
	v = \xi(\del_{x_0}+3x_1\del_{x_1}+\del_{s_0})
		+ \nu\del_{x_1}
\]
and thus the c-projective algebra of $g$ is 2-dimensional.
\smallskip

\noindent\textbf{Scenario 2 ($\dim(\homalg(h))=2$).}
If the homothetic algebra is 2-dimensional, due to Lemma~\ref{la:D2b.D3.2D-hom-algebra},
$$ u = \xi\del_{s_0}+\xi_1 w $$
with $C=-3$. We arrive at ($\xi,\xi_1,\nu\in\RR$)
$$ v^1
  = \xi_1\frac{k_1k_3}{3(k_1+3)}\,e^{-(k_1+3)s_0}\,
  			\sin(k_1s_1+k_2)^{-\frac{k_1+3}{k_1}}
  -3\xi\,x_1
  +\nu\,, $$
if $k_1\ne-3$, and
$$ v^1 = \xi_1\,\frac{k_3}{3}\,\bigg( 3s_0-\,\ln\sin(3s_1 -k_2) \bigg)
	-\xi\,(\beta+2)x_1+\nu $$
if $k_1=-3$.
The c-projective algebra of $g$ is therefore 3-dimensional. It is generated by
$$ \del_{x_1}\,,\qquad
	\del_{x_0}-3x_1\del_{x_1}+\del_{s_0} $$
and, if $k_1\ne-3$,
\[
	\frac{k_1k_3}{3(k_1+3)}\,e^{-(k_1+3)s_0}\,
		\sin(k_1s_1+k_2)^{-\frac{k_1+3}{k_1}}\del_{x_1}
	+ e^{-k_1s_0}\cos(k_1s_1+k_2)\del_{s_0}
	- e^{-k_1s_0}\sin(k_1s_1+k_2)\del_{s_1}\,.
\]
If $k_1=-3$, the third generator is instead
\[
	\frac{k_3}{3}\big( 3s_0-\ln\sin(3s_1-k_2) \big)\,\del_{x_1}
	\\
	+ e^{3s_0}\cos(3s_1-k_2)\del_{s_0}
	+ e^{3s_0}\sin(3s_1-k_2)\del_{s_1}\,.
\]
Note that w.l.o.g.~$\sin(k_1s_1+k_2)>0$ as per the comment before Lemma~\ref{la:D2b.D3.2D-hom-algebra}.
\smallskip

\noindent\textbf{Scenario 3 ($\dim(\homalg(h))=4$).}
In this case $h$ is flat and the homothetic algebra is analogous to that of  
Scenario~3 of the case D2b, where we formally replace $\beta=1$.
Proceeding as before, we find first
\begin{align*}
	v^1 &= \nu-3\xi_0\,x_1
	\\
	&\quad	+\xi_2\frac{
				(\mu_1^2-9)
				\cos(\tfrac12 \mu_1 s_0+\tfrac32 s_1)
				+6\mu_1\sin(\tfrac12\mu_1 s_0+\tfrac32 s_1)
			}{
				3(\mu_1^2+9)
			}\,
			\exp(\tfrac12\mu_1s_1-\tfrac32 s_0)
	\\
	&\quad	+\xi_3\frac{
			6\mu_1
			\cos(\tfrac12 \mu_1 s_0+\tfrac32 s_1)
			-(\mu_1^2-9)\sin(\tfrac12\mu_1 s_0+\tfrac32 s_1)
		}{
			3(\mu_1^2+9)
		}\,
		\exp(\tfrac12\mu_1s_1-\tfrac32 s_0)
\end{align*}
and then
\[
	v = \xi_0\del_{x_0}+v^1\del_{x_1}+u\,,
\]
where the parameters are $\xi_0,\xi_1,\xi_2,\xi_3,\nu\in\RR$.
The c-projective algebra of $g$ therefore is 5-dimensional.
This concludes the proof of Theorem~\ref{thm:degenerate.main}.

\section{Final remarks}
In the present paper we have obtained the full c-projective algebras for all
K\"ahler surfaces with essential c-projective vector fields, i.e.~for the
K\"ahler metrics described in~\cite{BMMR2015}.
We have, in particular, found that the c-projective vector fields in the case
of degenerate type metrics D1--D3 arise from c-projective vector fields
of a 2-dimensional metric $h$ involved in these metrics, which likely bears
some significance for practical applications. Moreover, this phenomenon
should be expected to arise also in higher dimension, which might be useful
for extending the results obtained here to higher dimensions.

We have seen in Theorem~\ref{thm:CHSC} that, while covering all metrics with
essential c-projective vector fields, the list of metrics in~\cite{BMMR2015}
does still contain metrics different from those of interest. Moreover, one
might argue that the list in~\cite{BMMR2015} describes the metrics only up to a c-projective transformation (and still in a non-sharp way). On another level,
one might ask for a description up to isometric transformations. The authors intend to address this problem in an upcoming paper, which will be
facilitated by the results obtained here.

\appendix
\section{Kähler metrics with constant HSC}

This appendix provides additional material that supplements the main body of the paper. It has two parts: we begin with an explicit realisation of the 
c-projective algebra of the Fubini-Study metric in complex dimension~2. The 
second part is dedicated to the c-projective equivalence of Kähler surfaces 
with constant HSC.

\subsection{Fubini-Study metric}\label{app:fubini-study}

In Example~\ref{ex:fubini-study} the c-projective algebra of the Fubini-Study
metric has been discussed. It is isomorphic to $\mathfrak{sl}(3,\CC)$.
Here we obtain a realisation of $\mathfrak{sl}(3,\CC)$ in terms of vector
fields on $\CC\mathbb{P}^2$.
By a straightforward computation, one obtains the c-projective connection
associated to $(\CC\mathbb{P}^2,J,g)$, given by
\begin{equation}\label{eqn:c.proj.conn.Fubini}
	\left\{
	\begin{array}{l}
		y_x^2s_{xx}-y_xs_xy_{xx}+t_xy_{xx}+s_{xx}=0
		\\
		-t_xy_xy_{xx}+y_x^2t_{xx}-s_xy_{xx}+t_{xx}=0
	\end{array}
	\right.
\end{equation}
in the coordinates of~\eqref{eqn:Fubini.metric.local}. Consider the second jet space $J^2(1,3)$ with coordinates  $$(x,y,s,t,y_x,s_x,t_x,y_{xx},s_{xx},t_{xx})\,.$$ Since we are working locally, we can think of $J^2(1,3)$ as $\mathbb{R}^{10}$. Thus, we can interpret \eqref{eqn:c.proj.conn.Fubini} as an $8$-dimensional variety in $J^2(1,3)\simeq\mathbb{R}^{10}$. It is well known that any vector field $X$ on (an open set of) $\mathbb{R}^4$ can be prolonged to a vector field $X^{(2)}$ on (an open set of) $J^2(1,3)$. The vector field $X$ is a \emph{point symmetry} of \eqref{eqn:c.proj.conn.Fubini} if its local flow sends solutions of \eqref{eqn:c.proj.conn.Fubini} into solutions: the condition for the vector field $X$ to be a point symmetry of \eqref{eqn:c.proj.conn.Fubini} is that $X^{(2)}$ vanishes on \eqref{eqn:c.proj.conn.Fubini}. Since c-projective vector fields are vector fields preserving $J$-planar curves (i.e., solutions to \eqref{eqn:c.proj.conn.Fubini}), they coincide with the set of the point symmetries of \eqref{eqn:c.proj.conn.Fubini}. 
The point symmetries of this system can be studied using Lie symmetry
techniques \cite{olver2000applications}, and one thus obtains the generators of
its c-projective algebra. These are the following $16$ vector
fields:
\begin{equation*}
	\begin{array}{llll}
		(x^2-y^2)\partial_x +2xy\partial_y + (xs-yt)\partial_s +
		(xt-ys)\partial_t \,\,, & \, y\partial_x-x\partial_y \,\,, & \,
		x\partial_x+y\partial_y \,\,, & \, \partial_x\,,
		\\
		\\
		-2xy\partial_x +(x^2-y^2)\partial_y - (xt+ys)\partial_s +
		(xs-yt)\partial_t \,\,, & \, y\partial_s-x\partial_t \,\,, & \,
		x\partial_s+y\partial_t \,\,, & \, \partial_y\,,
		\\
		\\
		-(xt+ys)\partial_x + (xs-yt)\partial_y - 2st\partial_s +
		(s^2-t^2)\partial_t \,\,, & \, t\partial_s-s\partial_t \,\,, & \,
		s\partial_s+t\partial_t \,\,, & \, \partial_s\,,
		\\
		\\
		(xs-yt)\partial_x + (xt+ys)\partial_y + (s^2-t^2)\partial_s +
		2st\partial_t \,\,, & \, t\partial_x-s\partial_y \,\,, & \,
		s\partial_x+t\partial_y \,\,, & \, \partial_t\,.
	\end{array}
\end{equation*}
These 16 vector fields generate the c-projective algebra
of~\eqref{eqn:Fubini.metric.local}. The c-projective algebra of any Kähler
metric of constant HSC is isomorphic to it.

\subsection{C-projective equivalence of Kähler metrics with constant HSC}

Here we provide some background material regarding the statement of 
Fact~\ref{fct:CHSC}. The following proposition seems to be considered 
folklore. Yet we have not been able to retrieve the original reference, and do 
not know if such an initial reference is in existence. According to hearsay a 
formal proof may exist, and might have been published in a Japanese journal 
in the second half of the 20th century. In fact, the proof is straightforward 
if we impose the additional assumption of Riemannian signature.
\begin{proposition}
	Let $(M,J,g)$ be a K\"ahler surface of constant HSC~$\kappa$, where $g$ 
	is of arbitrary signature.
	Then, locally, it is c-projectively equivalent to $\mathbb{C}\mathbb{P}^2$ 
	with the Fubini-Study metric.
\end{proposition}
\begin{proof}
	Let us assume first that $g$ were a Kähler metric as in the hypothesis, 
	but specifically of Riemannian signature.
	Then it is a well-known fact that~$g$ is locally isometric to (see 
	\cite{GV1979,Bochner1947}),
	\begin{itemize}
		\item
		if $\kappa>0$, to a multiply of the Fubini-Study metric,
		\begin{equation}\label{eqn:Fubini.metric.local.alt}
			\frac{4}{\kappa}
			\frac{
				(1+\sum_iz_i\bar{z}_i)\sum_jdz_jd\bar{z}_j
				-\sum_{i,j}\bar{z}_jz_idz_jd\bar{z}_i
			}{
				(1+\sum_iz_i\bar{z}_i)^2
			}\,,
		\end{equation}
		\item
		if $\kappa<0$, to a multiple of the Bergman metric,
		\begin{equation}\label{eqn:Bergman.metric.local} 
			-\frac{4}{\kappa}
			\frac{
				(1-\sum_iz_i\bar{z}_i)\sum_jdz_jd\bar{z}_j
				+\sum_{i,j}\bar{z}_jz_idz_jd\bar{z}_i
			}{
				(1-\sum_iz_i\bar{z}_i)^2
			}\,,
		\end{equation}
		\item
		if $\kappa=0$, to the Euclidean metric,
		\begin{equation}\label{eqn:eucl.metric.local}
			\sum_idz_i d\bar{z}_i\,.
		\end{equation}
	\end{itemize}
	Let us now turn to non-Riemannian signature. An inspection of the proof 
	in~\cite{Bochner1947} shows that the signature does not play a crucial 
	role in it. One thereby verifies the following, more general statement.
	Let $g$ be a K\"ahler metric as in the hypothesis, of arbitrary 
	signature. Then it is locally isometric,
	\begin{itemize}
		\item if $\kappa>0$, to
		\begin{equation}\label{eqn:Fubini.metric.local.mod} 
		\frac{4}{\kappa}
		\frac{(1+\sum_i\epsilon_iz_i\bar{z}_i)\sum_j\epsilon_jdz_jd\bar{z}_j- 
		\sum_{i,j}\epsilon_j\epsilon_i\bar{z}_jz_idz_jd\bar{z}_i}{(1+\sum_i\epsilon_iz_i\bar{z}_i)^2}\,,
		 \quad \text{for some}\,\, \epsilon_i\in\{-1,1\}\,,
		\end{equation}
		\item if $\kappa<0$, to
		\begin{equation}\label{eqn:Bergman.metric.local.mod} 
		-\frac{4}{\kappa}
		\frac{(1-\sum_i\epsilon_iz_i\bar{z}_i)\sum_j\epsilon_jdz_jd\bar{z}_j 
		+\sum_{i,j}\epsilon_j\epsilon_i\bar{z}_jz_idz_jd\bar{z}_i}{(1-\sum_i\epsilon_iz_i\bar{z}_i)^2}\,,
		 \quad \text{for some}\,\, \epsilon_i\in\{-1,1\}\,,
		\end{equation}
		\item if $\kappa=0$, to
		\begin{equation}\label{eqn:eucl.metric.local.mod}
			\sum_i\epsilon_idz_i d\bar{z}_i\,, \quad \text{for some}\,\, 
			\epsilon_i\in\{-1,1\}\,.
		\end{equation}
	\end{itemize}
	Note that the metrics \eqref{eqn:Fubini.metric.local.mod}, 		
	\eqref{eqn:Bergman.metric.local.mod} and 		
	\eqref{eqn:eucl.metric.local.mod} indeed are generalisations 
	of~\eqref{eqn:Fubini.metric.local.alt}, \eqref{eqn:Bergman.metric.local} 
	and~\eqref{eqn:eucl.metric.local}, respectively.
	The signature is determined by the constants $\epsilon_i$.
	
	The crucial observation now is that the metrics above are c-projectivly 
	equivalent to each other. This can be confirmed by checking that they 
	share the same c-projective connection. Indeed, in coordinates $(x,y,s,t)$ where $z_1=x+iy$ and $z_2=s+it$, the c-projective connection of~\eqref{eqn:Fubini.metric.local.mod}, \eqref{eqn:Bergman.metric.local.mod} 
	and~\eqref{eqn:eucl.metric.local.mod} is identical 
	to~\eqref{eqn:c.proj.conn.Fubini}.
	This proves the claim.
\end{proof}
We remark that the statement of the proposition implies that the c-projective 
Lie algebras of all the 
metrics~\eqref{eqn:Fubini.metric.local.alt}--\eqref{eqn:eucl.metric.local.mod}
are isomorphic to $\mathfrak{sl}(3,\mathbb{C})$.

\section{Companion metrics for the metrics in Fact~\ref{fct:metrics}}
\label{app:companion.metrics}
The proofs of Theorems \ref{thm:CHSC}\,--\,\ref{thm:degenerate.main} do not
exclusively use the metrics given in Fact \ref{fct:metrics}, but rely on the
tensor $L$ introduced in Equation~\eqref{eqn:L}. Such $(1,1)$-tensors are
related to c-projectively equivalent metrics by the following proposition,
see~\cite{domashev_mikesh_1978} and also~\cite[Section~5]{CEMN2020}.
\begin{proposition}[\cite{domashev_mikesh_1978}]
	Two K\"ahler metrics (of arbitrary signature) $g, \hat{g}$ on a manifold
	$(M,J)$ (of real dimension $2n$) are c-projectively equivalent if and
	only if the tensor
	\begin{equation*}
		L^i_j = \left| \frac{\det \hat{g}}{\det g}
		\right|^{\frac{1}{2(n+1)}}\hat{g}^{il}g_{lj}
	\end{equation*}
	satisfies the equation
	\begin{equation}\label{eqn:nab.L}
		\nabla_k L_{ij} =
		\Lambda_i g_{jk} + \Lambda_j g_{ik} + \bar{\Lambda}_i \omega_{jk} +
		\bar{\Lambda}_j \omega_{ik}
	\end{equation}
	where
	\begin{equation}\label{eqn:nab.L.defs}
		\bar{\Lambda}_i = J\indices{^a_i} \Lambda_a
		\,, \qquad
		\omega_{ij} = J\indices{^a_i} g_{aj}
		\,,\qquad
		\Lambda_i = \nabla_i \Lambda
		\,, \qquad
		\Lambda = \frac{1}{4} \tr(L)\,.
	\end{equation}
\end{proposition}
\noindent Note that the last relation in~\eqref{eqn:nab.L.defs} is not a
definition, but a consequence of~\eqref{eqn:nab.L}.
The prolongation equations for~\eqref{eqn:nab.L} are called \emph{Sinjukov
	equations}. The (special) Sinjukov equations~\eqref{eqn:sinjukov} are a
special case of these.\medskip

For the proofs of Theorems \ref{thm:CHSC}\,--\,\ref{thm:degenerate.main}, the
tensor $L$ is obtained using metrics $\hat{g}$ c-projectively equivalent to
the metrics in Fact~\ref{fct:metrics}. The explicit metrics $\hat{g}$ used in
each case are specified below. They have been taken
from~\cite[Theorem~1.5]{BMMR2015}. We mention also \cite{BMR2021}, where an
extension to higher dimension can be found. Note that the functions $\rho, F$
etc.~below are the same as the respective objects in Fact~\ref{fct:metrics}.

Via~\eqref{eqn:metric.construction}, a family of c-projectively equivalent
metrics generated by $g,\hat{g}$ can be constructed.
In the case that the metric $g$ is of non-constant HSC, this is the whole
c-projective class $[g]$. The reason is that Lemma~\ref{la:mobility} in this
case ensures that the degree of mobility is $D(g)=2$. For an explicit,
parametrised form of these families, for the metrics of
Fact~\ref{fct:metrics}, see also~\cite[Theorems 1.2 and 1.5]{BMMR2015}.
\subsection{Liouville type metrics}
\begin{equation*}
	\begin{split}
		\hat{g} = & \frac{1}{\rho_0^2 \rho_1^2 (\rho_0 - \rho_1)}
		\left[
		(\rho_0 - \rho_1)^2 \rho_0 \rho_1 \left(\frac{F_0^2}{\rho_0} dx_0^2 + 
		\epsilon \frac{F_1^2}{\rho_1} dx_1^2\right)
		\right.\\
		&\left.
		+
		\left( \frac{\rho_0'}{F_0}\right)\rho_1
		(ds_0 + \rho_0 ds_1)^2
		+\varepsilon \left( \frac{\rho_1'}{F_1}\right)\rho_0 (ds_0 + \rho_0 
		ds_1)^2
		\right]
	\end{split}
\end{equation*}

\subsection{Complex type metrics}
\begin{equation*}
	\begin{split}
		\hat{g} = & \frac{1}{\rho^2 \bar{\rho}^2 (\bar{\rho} - \rho)}
		\left[
		\frac{1}{4}
		(\bar{\rho} - \rho)^2 \rho \bar{\rho} \left(\frac{F^2}{\rho} dz^2 + 
		\epsilon \frac{\bar{F}^2}{\bar{\rho}} d\bar{z}^2\right)
		\right.\\
		&\left.
		+
		\left( \frac{1}{\bar{F}}\frac{d\bar{\rho}}{d\bar{z}}\right)^2\rho
		(ds_0 + \bar{\rho} ds_1)^2
		- \left( \frac{1}{F}\frac{d\rho}{dz}\right)^2\bar{\rho}
		(ds_0 + \rho ds_1)^2
		\right]
	\end{split}
\end{equation*}

\subsection{Degenerate type metrics}
\begin{equation*}
	\begin{split}
		\hat{g} = & \frac{1}{\rho-1}
		\left(
		\rho h
		+\frac{\rho F^2}{\rho-1} dx^2
		+ \frac{1}{\rho (\rho-1)}\left(\frac{\rho'}{F}\right)^2 \theta^2
		\right)
	\end{split}
\end{equation*}

\section*{Acknowledgements}
The authors thank Vladimir Matveev and Stefan Rosemann for insightful
discussions, remarks and for providing helpful references.
The authors acknowledge support through the project PRIN 2017 ``Real and
Complex Manifolds: Topology, Geometry and holomorphic dynamics'', the project
``Connessioni proiettive, equazioni di Monge-Amp\`ere e sistemi integrabili''
Istituto Nazionale di Alta Matematica (INdAM) and the MIUR grant
``Dipartimenti di Eccellenza 2018-2022 (E11G18000350001)''.
GM acknowledges the ``Finanziamento alla Ricerca
(53{\_}RBA17MANGIO)''.
GM is a member of GNSAGA of INdAM; AV acknowledges previous membership.
JS was a research fellow of INdAM.

\printbibliography

\end{document}